\UseRawInputEncoding
\documentclass[12pt,reqno]{article}

\usepackage{amsfonts}
\usepackage{mathrsfs}

\usepackage{amsmath,amsthm,amsfonts,amssymb,latexsym,mathrsfs,color}

\usepackage[colorlinks=true,
linkcolor=webblue, filecolor=webbrown, citecolor=webred ]{hyperref}
\definecolor{webblue}{rgb}{0,.5,0}
\definecolor{webred}{rgb}{0,.5,0}
\definecolor{webbrown}{rgb}{.6,0,0}

\newtheorem{thm}{Theorem}[section]
\newtheorem{lem}[thm]{Lemma}

\newtheorem{prop}[thm]{Proposition}
\newtheorem{conj}[thm]{Conjecture}
\newcommand{\qbinom}[2]{{#1\brack#2}}
\newtheorem{cla}{Claim}

\newtheorem{de}{Definition}

\theoremstyle{definition}

\newtheorem{rem}[thm]{Remark}

\numberwithin{equation}{section}

\numberwithin{equation}{section}
\newcommand{\qbinomd}[3]{\genfrac{[}{]}{0pt}{}{#1}{#2}_{q^{#3}}}
\newcommand{\qint}[1]{[#1]_{q}}
\newcommand{\Z}{\mathbb Z}
\newcommand{\Q}{\mathbb Q}

\newcommand{\wt}{wt}
\newcommand{\rad}{rad}

\title{Proof of a nonnegativity conjecture of Prellberg--Stanton\thanks{This work was supported in part by the National Key R\&D Program of China (No.~2024YFA1013903) and National Natural Science Foundation of China (No.~12471310).}}
\author{ Hong-Zhang Ai, Lili Mu, Bao-Xuan Zhu\footnote{Corresponding author\newline\hspace*{5mm}
{\it Email address:} {aihz@jsnu.edu.cn (H.-Z. Ai), lilimu@jsnu.edu.cn (L.L. Mu),  bxzhu@jsnu.edu.cn (B.-X. Zhu)}}}
\date{\footnotesize
School of Mathematics and Statistics, Jiangsu Normal University,
Xuzhou 221116, P.R. China}

\begin{document}

\maketitle

\begin{abstract}
In 1994, Friedman, Joichi and Stanton conjectured that coefficients of the power-series $q+(1-q)\prod_{i=n}^{2n-1}\frac{1}{1-q^{i}}$ are nonnegative
for $n\geq3$ and $n$ being odd. Let $\qbinom nm$ denote the Gaussian polynomial. In 2003, Prellberg and Stanton conjectured that coefficients of the power-series
$(1-q)\qbinom nm\prod_{i=0}^{m-1}\frac{1}{1-q^{i+m}}$ are nonnegative
for $0<m<n$, except when $n$ is even and $m\in\{2,n-2\}$. The latter conjecture implies the former.
In this paper, we prove a stronger result concerning coefficientwise nonnegativity of another power series, which implies the conjecture of Prellberg and Stanton. The proof uses a M\"obius correction from Gaussian polynomials, weight-preserving injections on partitions, explicit linearly independent vectors in the kernel of a derivation operator and the KOH identity.
\end{abstract}
\tableofcontents
\section{Introduction}
Throughout this paper, let \(q\) be an indeterminate. Recall the notation
\begin{equation}\label{def+Gauss+poly}
 \qint{n}=\frac{1-q^n}{1-q},\qquad (a;q)_n=\prod_{i=0}^{n-1}(1-aq^{i}),\qquad  \qbinom nk=\frac{(q;q)_n}{(q;q)_k(q;q)_{n-k}}.
\end{equation}
The last expression is the Gaussian polynomial \cite[Section~1.7]{StanleyEC1}.
For two formal power series \(A(q)\) and \(B(q)\), we write \(A(q)\geq_q B(q)\) when every coefficient
of the formal power series \(A(q)-B(q)\) is nonnegative, and let
\(A_k=[q^k]A(q)\) with \(A_k=0\) for \(k<0\).

In the theory of partitions and $q$-series, it is a basic problem to determine the signs of coefficients of the power-series. 
A prescribed sign pattern often hints some combinatorial insight, in particular, encoding partition inequalities.
For instance, Borwein Conjecture \cite{And93} says that coefficients of the formal power-series
$(q;q^3)_n(q^2;q^3)_n$
have the repeating sign pattern $+,-,-$, with zero allowed in either sign class. It can be reformulated as the existence of an injection from partitions. 
Wang~\cite{Wang2022} proved this conjecture by analytic methods.
In \cite{FJS1994}, Friedman, Joichi and Stanton made the following conjecture.
\begin{conj}\label{conj+FJS}
Let $n\geq3$ and $n$ is odd. Then all coefficients of the formal power series $q+\frac{1-q}{(q^n;q)_n}$ are nonnegative.
\end{conj}
The nonnegativity in  Conjecture \ref{conj+FJS} implies increasing
partition functions on the set $A=\{n,n+1,\ldots,2n-1\}$ \cite{FJS1994}.
 For $n\leq99$, Conjecture \ref{conj+FJS} was checked by a computer proof (see \cite{FJS1994,Han}). Andrews \cite{Andrews} showed Conjecture \ref{conj+FJS} for prime values of $n$. In 2003, Prellberg and Stanton \cite{PS2003} confirmed Conjecture \ref{conj+FJS} using an
identity for the rational function. Prellberg and Stanton noted a further identity:
\begin{equation}\label{eq:PS-decomposition}
 \frac{1-q}{(q^n;q)_n}+q
 =\frac{1}{1-q^{n(2n-1)}}
 \left(1-q^{n(2n-1)+1}
 +\sum_{m=1}^{n-1}q^{m(m+n-1)}
 \qbinom{n}{m}\frac{1-q}{(q^n;q)_m}\right),
\end{equation}
which would also imply Conjecture \ref{conj+FJS} if the individual terms are nonnegative.
Thus, they posed the following nonnegativity conjecture.

\begin{conj}\cite{PS2003}\label{conj+PS}
Let \(n,m\) be integers with \(n>0\) and \(0<m<n\). Then
\[
 \qbinom{n}{m}\frac{1-q}{(q^n;q)_m}\geq_q0
\]
if $n$ is odd, or  even and
$m\notin\{2,n-2\}$.
\end{conj}
Twenty years have passed, Conjecture \ref{conj+PS} is still open. In this paper, we prove the following stronger result for nonnegativity.
\begin{thm}\label{thm:strong}
Let $m$ and $n$ be positive integers. If $3\le m\le n-3$, then 
$$\qbinom nm
       \frac{1-q}{(1-q^n)(1-q^{n+1})}\ge_q0.$$
\end{thm}
As a corollary of Theorem \ref{thm:strong}, we confirm Conjecture \ref{conj+PS}.

\begin{thm}\label{thm:main}
Let $n$ and $m$ be integers with $0<m<n$. Then \[
 \qbinom{n}{m}\frac{1-q}{(q^n;q)_m}\geq_q0
\] if $n$ is odd, or  even and
$m\notin\{2,n-2\}$.
\end{thm}

For convenience, define
\begin{equation}\label{def+Fnm+Gnm}
 F_{n,m}(q)=\qbinom nm\frac{1-q}{(q^n;q)_m},
 \qquad
 G_{n,m}(q)=\qbinom nm
       \frac{1-q}{(1-q^n)(1-q^{n+1})}
\end{equation}
when $n,m\in\Z$ and $0<m<n$. So, for $m\ge2$, we have
\begin{equation}\label{eq:f-to-g}
 F_{n,m}(q)=\frac{G_{n,m}(q)}{(q^{n+2};q)_{m-2}}.
\end{equation}
Obviously, for two formal power series $A(q)$ and $B(q)$, $A(q)B(q)\geq_q0$ if $A(q)\geq_q0$ and $B(q)\geq_q0$. 
Every additional reciprocal factor has nonnegative coefficients, so
by (\ref{eq:f-to-g}), Theorem~\ref{thm:strong} implies Theorem~\ref{thm:main} in the interior range.
The boundary cases are elementary and are included in
Section~\ref{sec:completion}; they also show that the stated exceptions are
necessary.

To describe the obstruction of the proof, set $P(q)=\qbinom{n-1}{m-1}$. By (\ref{def+Gauss+poly}) and (\ref{def+Fnm+Gnm}), we have
\begin{equation}\label{eq:Gnm}
  G_{n,m}(q)=\frac{P(q)}{\qint m(1-q^{n+1})}.
\end{equation}

Let $\gcd(n,m)$ denote the greatest common divisor of integers $n$ and $m$. For $\gcd(n,m)=1$, the quotient $P(q)/\qint m$ is a polynomial with nonnegative
coefficients (see Andrews \cite[Theorem~2]{Andrews} and
Haiman \cite[Section~2.5]{Haiman}). We include a proof of this case in Lemma~\ref{lem:coprime} for
completeness. The central issue is the case $\gcd(n,m)>1$. In such case, the polynomial $P(q)$ does not vanish at all the roots of
$\qint m$, so the same argument does not apply. In the noncoprime interior
range, we first use the symmetry of Gaussian polynomials to reduce $m\in[3, n-3]$ to
$m\in[3, n/2]$ and construct a polynomial $R(q)$ (see (\ref{def+R})) in terms of Gaussian polynomials and the M\"obius inversion, called {\it the M\"obius correction}, whose values at the nontrivial \(m\)-th roots of
unity remove the obstruction to dividing a Gaussian polynomial by
\(\qint{m}\) and make
\begin{equation}\label{eq:Q+poly}
Q(q):= \frac{P(q)-\qint{n+1}R(q)}{\qint m}\in\Z[q].
\end{equation}
Then by (\ref{eq:Gnm}) and (\ref{eq:Q+poly}), we have
\begin{equation}\label{eq:G-to-R+Q}
 G_{n,m}(q)=\frac{R(q)}{1-q^m}+\frac{Q(q)}{1-q^{n+1}}.
\end{equation}

  In Section \ref{sec:reduction}, we reduce the desired positivity $G_{n,m}(q)\geq_q0$ to three coefficient positivity inequalities:
$R(q)\ge_q0$, $W(q)\ge_q0$ and $[q^k]((1-q)P(q)-R(q))\geq0$ for $0\le k\le\delta/2$, where $W(q)$ is another polynomial. In order to prove $R(q)\geq_q0$, in Section \ref{sec:correction-positive}, we first rewrite $R(q)$ in terms of partitions (see (\ref{def+Rq+new})). In order to account for the signs
in $R(q)$, we introduce a splitting partition transformation (\ref{eq:partition-split}) of partition and show it to be a weight-preserving injection. 
The injection implies positivity of the coefficients of $R(q)$. For $(1-q)P(q)\ge_q R(q)$, we need coefficient comparisons. We first use the kernel formula to present $[q^k](1-q)P(q)$ and introduce a method to construct linear independent kernel vectors from partitions in Section \ref{sec:kernel}. Then, in Section \ref{sec:domination}, in order to prove $(1-q)P(q)\ge_q R(q)$, it suffices to construct linear independent vectors and use the KOH identity \cite{OHara1990,Zanello2015,Zeilberger1989}. The process needs splitting partition transformation (\ref{eq:partition-split}) and the construction method for linear independent kernel vectors in Section \ref{sec:kernel}. We arrange the proof of $W(q)\geq_q0$ in Section~\ref{sec:reciprocal-positive} and the proofs of main results in Section~\ref{sec:completion}.

\section{Gaussian polynomials and explicit kernel vectors}
\label{sec:kernel}

For $n\in\Z_{\ge0}$, a partition of $n$ is a
sequence $\lambda=(\lambda_1,\ldots,\lambda_t)$ of integers
satisfying $\sum_{i=1}^{t} \lambda_i=n$ and
$\lambda_1\geq\lambda_2\geq\cdots\geq\lambda_t\geq 1$ when $t>0$.
If $\lambda$ is a partition of $n$, then we write $\lambda\vdash n$ and define its length
and weight by
\[
 \ell(\lambda)=t,\quad
 \wt(\lambda)=\sum_{j=1}^{t}\lambda_j.
\]
The unique partition of $0$ is the empty sequence, which has length
and weight zero. We put $\lambda_j=0$ for $j>\ell(\lambda)$.
Repeated parts may be recorded by exponents: for example, $(i^t)$
denotes the partition with $t$ parts being $i$.

Let integers $a\ge1$ and $b\ge0$.  We write
\[
 \mathcal G_{a,b}(q)=\qbinom{a+b}{b}.
\]
From the partition interpretation of Gaussian polynomials
\cite[Proposition~1.7.3]{StanleyEC1}, we obtain for every integer $k\ge0$ that the coefficient
\begin{equation}\label{eq:coeff+Gauss}
 [q^k]\mathcal G_{a,b}(q)
 =\bigl|\{\lambda\vdash k:\ell(\lambda)\le b,\ \lambda_1\le a\}\bigr|.
\end{equation}
It follows from the symmetry of Gaussian polynomials \cite[Section~1.7]{StanleyEC1} that
\begin{equation}\label{eq:kernel-reciprocity}
 q^{ab}\mathcal G_{a,b}(q^{-1})=\mathcal G_{a,b}(q).
\end{equation}

The unimodality of Gaussian polynomials has been proved by several
methods; see,
for example, \cite{ChenJia2022} for an account of the classical
Cayley--Sylvester kernel formula, Stanley's proof~\cite{Stanley1980} in terms of the hard Lefschetz theorem, O'Hara's constructive
proof~\cite{OHara1990}, Zeilberger's KOH identity~\cite{Zeilberger1989}, Kronecker coefficients of Pak and Panova~\cite{PakPanova2014,PakPanova2017,PakPanovaVallejo2016}, and so on.

\subsection{The kernel formula}

In this subsection, we will give a self-contained proof of
the kernel formula, which in fact recovers the unimodality and
fixes the operators and normalization used in the subsequent
constructions. The explicit kernel vectors constructed below will yield
the stronger coefficient bounds required in
Section~\ref{sec:domination}.

Fix a positive integer $a$ and work in $\Q[x_0,\ldots,x_a]$.
For $\boldsymbol{\alpha}=(\alpha_0,\ldots,\alpha_a)\in\Z_{\ge0}^{a+1}$, set
$\textbf{x}^{\boldsymbol{\alpha}}=\prod_{i=0}^{a}x_i^{\alpha_i}$.
Its degree is $\sum_{i=0}^{a}\alpha_i$, and its
index weight is $\sum_{i=0}^{a}i\alpha_i$.
For $b\in\Z_{\ge0}$ and $k\in\Z$, define
\[
 V_{b,k}
 =\operatorname{span}_{\Q}
 \left\{\textbf{x}^{\boldsymbol{\alpha}}:
    \sum_{i=0}^{a}\alpha_i=b,\quad
    \sum_{i=0}^{a}i\alpha_i=k\right\}.
\]
The span of the empty set is understood to be $\{0\}$; in
particular, $V_{b,k}=\{0\}$ for $k<0$ or $k>ab$.
The displayed monomials form a basis of $V_{b,k}$. Each monomial
$\textbf{x}^{\boldsymbol{\alpha}}\in V_{b,k}$ determines a partition with
exactly $\alpha_i$ parts of size $i$ for each $1\le i\le a$. It has weight $k$,
at most $b$ parts, and all parts at most $a$. Conversely, such a
partition determines a monomial with $\alpha_0=b-\ell(\lambda)$. So this
correspondence is a bijection. In consequence, by the coefficient formula (\ref{eq:coeff+Gauss}),
we have
\begin{equation}\label{eq:kernel-dimension}
 \dim_{\Q}V_{b,k}=[q^k]\mathcal G_{a,b}(q).
\end{equation}

Following \cite{ChenJia2022}, we define derivations of
$\Q[x_0,\ldots,x_a]$ by
\begin{eqnarray}
 \mathcal D&=&\sum_{i=1}^{a}i x_{i-1}\frac{\partial}{\partial x_i},\quad\quad
 \mathcal E=\sum_{i=0}^{a-1}(a-i)x_{i+1}\frac{\partial}{\partial x_i},\\
 \mathcal H&=&\sum_{i=0}^{a}(a-2i)x_i\frac{\partial}{\partial x_i},
 \end{eqnarray}
where $x_{-1}=x_{a+1}=0$. It is obvious that these three derivations are linear operators.
Form these definitions, it is easy to show
$\mathcal D(V_{b,k})\subseteq V_{b,k-1}$ and
$\mathcal E(V_{b,k})\subseteq V_{b,k+1}$.
For each generator $x_i$, computation gives
\[\mathcal D\mathcal Ex_i-
\mathcal E\mathcal Dx_i=
 \bigl((a-i)(i+1)-i(a-i+1)\bigr)x_i=(a-2i)x_i.
\]
Therefore  we have
\begin{eqnarray}
 \mathcal D\mathcal E-\mathcal E\mathcal D=\mathcal H.
\end{eqnarray}
Moreover, for every basis monomial $x^{\boldsymbol{\alpha}}\in V_{b,k}$, we also obtain
\begin{eqnarray}
 \mathcal H\,\textbf{x}^{\boldsymbol{\alpha}}
 &=&\left(\sum_{i=0}^{a}(a-2i)\alpha_i\right)\textbf{x}^{\boldsymbol{\alpha}}
 =(ab-2k)\textbf{x}^{\boldsymbol{\alpha}}.
\end{eqnarray}

We now use these operators to give the kernel formula, which also implies the unimodality of Gaussian
polynomials.

\begin{lem}\label{lem:kernel}
Let $a$ and $b$ be positive integers. For every integer $k\in [0,\ ab/2]$, we have
\begin{eqnarray}\label{formal+kernel}
 \dim_{\Q}\ker\bigl(\mathcal D:V_{b,k}\longrightarrow V_{b,k-1}\bigr)
 &=&[q^k](1-q)\mathcal G_{a,b}(q).
\end{eqnarray}
In particular, $\mathcal G_{a,b}(q)$ is unimodal.
\end{lem}

\begin{proof}
We will first prove that (\ref{formal+kernel}) holds. For $k=0$, $[q^0](1-q)\mathcal G_{a,b}(q)=1$ and it follows from $V_{b,0}=\Q x_0^b$ and $\mathcal D x_0^b=0$ that
the dimension of the kernel is one. So (\ref{formal+kernel}) is obvious for $k=0$. In the following it suffices to prove that
$\mathcal D:V_{b,k}\to V_{b,k-1}$ is surjective for
$1\le k\le ab/2$. Indeed, combining surjectivity and
\eqref{eq:kernel-dimension} yields
\[
 \begin{aligned}
 \dim\ker\mathcal D
 &=\dim V_{b,k}-\dim V_{b,k-1}\\
 &=[q^k]\mathcal G_{a,b}(q)-[q^{k-1}]\mathcal G_{a,b}(q)\\
 &=[q^k](1-q)\mathcal G_{a,b}(q).
 \end{aligned}
\]

We now present the proof of the required surjectivity. In the following proof,
we extend the domain $\mathbb Q[x_0,\ldots,x_a]$ of $\mathcal D,\mathcal E,\mathcal H$ and $V_{b,k}$ onto $\mathbb R[x_0,\ldots,x_a]$.
Since the corresponding matrix of the linear operator $\mathcal D$ acting on monomial bases has rational entries, this extension does not change its
rank and kernel dimension. Thus surjectivity and the dimension of its kernel are unaffected by this
change of scalars.

For any $\textbf{x}^{\boldsymbol{\alpha}}\in V_{b,k}$ and $\textbf{x}^{\boldsymbol{\beta}}\in V_{b,k}$, define a bilinear inner product by
\[
 \langle \textbf{x}^{\boldsymbol{\alpha}},\textbf{x}^{\boldsymbol{\beta}}\rangle
 :=\begin{cases}
 \displaystyle\left(\prod_{i=0}^{a}\alpha_i!\right)
   \prod_{i=0}^{a}\left(\frac{i!(a-i)!}{a!}\right)^{\alpha_i},
   &\boldsymbol{\alpha}=\boldsymbol{\beta},\\
 0,&\boldsymbol{\alpha}\ne\boldsymbol{\beta}.
 \end{cases}
\]
Obviously, this inner product is nonnegative.
We write $\|f\|^2=\langle f,f\rangle$.
We first give the next claim.

\begin{cla} The inner product $\langle\cdot,\cdot\rangle$ on $V_{b,k}$ makes $\mathcal E$ adjoint to $\mathcal D$, i.e.,
\begin{eqnarray}\label{inner+product+adjoint}
\langle\mathcal D f,g\rangle=\langle f,\mathcal E g\rangle
\end{eqnarray}
for all $f,g$ in $V_{b,k}$. 
\end{cla}

\begin{proof}
To obtain this identity, we will assign positive
weights to the monomial basis. Let $\textbf{x}^{\boldsymbol{\alpha}}\in V_{b,k}$ be an arbitrary monomial with $\alpha_i>0$
for $1\le i\le a$. And assume that
$\boldsymbol{\beta}=\boldsymbol{\alpha}-\boldsymbol{e}_i+\boldsymbol{e}_{i-1}$, where $\boldsymbol{e}_i$ is the coordinate
vector corresponding to $x_i$. Then in terms of the definition of the inner product $\langle\cdot,\cdot\rangle$, we obtain
\begin{eqnarray}\label{form+two+vectors}
\frac{\|\textbf{x}^{\boldsymbol{\beta}}\|^2}{\|\textbf{x}^{\boldsymbol{\alpha}}\|^2}
 =\frac{\alpha_{i-1}+1}{\alpha_i}\frac{a-i+1}{i},
\end{eqnarray}
In addition, it follows from the definitions of the operator $\mathcal D$ and the inner product $\langle\cdot,\cdot\rangle$ that
\begin{eqnarray}
\langle\mathcal D \textbf{x}^{\boldsymbol{\alpha}},\textbf{x}^{\boldsymbol{\beta}}\rangle
 &=&i\,\alpha_i\,\|\textbf{x}^{\boldsymbol{\beta}}\|^2,\label{inner+product+adjoint+1}\\
 \langle \textbf{x}^{\boldsymbol{\alpha}},\mathcal E \textbf{x}^{\boldsymbol{\beta}}\rangle
 &=&(a-i+1)(\alpha_{i-1}+1)\,\|\textbf{x}^{\boldsymbol{\alpha}}\|^2.\label{inner+product+adjoint+2}
\end{eqnarray}
Thus, by (\ref{form+two+vectors})-(\ref{inner+product+adjoint+2}), we obtain
\begin{eqnarray}
\langle\mathcal D \textbf{x}^{\boldsymbol{\alpha}},\textbf{x}^{\boldsymbol{\beta}}\rangle
 &=& \langle \textbf{x}^{\boldsymbol{\alpha}},\mathcal E \textbf{x}^{\boldsymbol{\beta}}\rangle
\end{eqnarray}
for any $\textbf{x}^{\boldsymbol{\alpha}}\in V_{b,k}$ with all $\alpha_i>0$ and $\boldsymbol{\beta}=\boldsymbol{\alpha}-\boldsymbol{e}_j+\boldsymbol{e}_{j-1}$, where $\boldsymbol{e}_j$ is the coordinate
vector corresponding to $x_j$. In addition, for any pair of monomials not related in this way, their inner
product is zero. Therefore, using the bilinearity of the inner product gives the identity (\ref{inner+product+adjoint}) for all $f,g$ in $V_{b,k}$. This prove this claim.
\end{proof}

Fix $1\le k\le ab/2$. Let $y\in V_{b,k-1}$ be orthogonal
to the image of $\mathcal D:V_{b,k}\to V_{b,k-1}$.
For every $x\in V_{b,k}$, adjointness (\ref{inner+product+adjoint}) gives
$\langle x,\mathcal E y\rangle=\langle\mathcal D x,y\rangle=0$.
Since $\mathcal E y\in V_{b,k}$, taking $x=\mathcal E y$
we have
$\langle\mathcal D\mathcal E y,y\rangle=0$.
On the other hand, using $\mathcal D\mathcal E-\mathcal E\mathcal D=\mathcal H$,
$\mathcal H y=(ab-2k+2)y$, and adjointness, we obtain
\[
 \begin{aligned}
 \langle\mathcal D\mathcal E y,y\rangle
  &=\langle(\mathcal E\mathcal D +\mathcal H) y,y\rangle\\
  &=\langle\mathcal E\mathcal D y,y\rangle
    +\langle\mathcal H y,y\rangle\\
  &=\langle\mathcal D y,\mathcal D y\rangle
    +\langle(ab-2k+2)y,y\rangle\\
  &=\|\mathcal D y\|^2+(ab-2k+2)\|y\|^2.
 \end{aligned}
\]
In consequence, we derive $\|\mathcal D y\|^2+(ab-2k+2)\|y\|^2=0$.
Here $ab-2k+2\ge2$. Thus, in terms of the nonnegativity of the inner product, we have $\|y\|^2=0$, which further implies $y=0$.
Thus $(\operatorname{im}\mathcal D)^\perp=\{0\}$ in $V_{b,k-1}$.
Since $V_{b,k-1}$ is finite-dimensional,
$\operatorname{im}\mathcal D=V_{b,k-1}$, so $\mathcal D$ is surjective.

Finally, in order to prove that the coefficient sequence of
$\mathcal G_{a,b}$ is unimodal, by \eqref{eq:kernel-reciprocity}, it suffices
to show that the coefficients are nondecreasing up to $ab/2$, which is immediate from the nonnegativity of $[q^k](1-q)\mathcal G_{a,b}(q)$ for $1\le k\le ab/2$ in the kernel formula (\ref{formal+kernel}). This completes the proof of this lemma.
\end{proof}

\subsection{Linearly independent  kernel vectors from partitions}

In this subsection, we will use partitions to construct linearly independent kernel vectors.

Let $\boldsymbol{\alpha},\boldsymbol{\beta}\in\Z_{\ge0}^{a+1}$ be distinct,
and put
$j=\max\{i\in\{0,\ldots,a\}:\alpha_i\ne\beta_i\}$.
Define the lexicographic order on the monomials of
$\Q[x_0,\ldots,x_a]$ by
\[
 \textbf{x}^{\boldsymbol{\alpha}}>\textbf{x}^{\boldsymbol{\beta}}
 \quad\Longleftrightarrow\quad \alpha_j>\beta_j.
\]
In particular, $x_a>x_{a-1}>\cdots>x_0$.
For a nonzero polynomial
$f=\sum_{\boldsymbol{\alpha}}c_{\boldsymbol{\alpha}}\textbf{x}^{\boldsymbol{\alpha}}$,
define its {\it leading monomial} to be the largest $\textbf{x}^{\boldsymbol{\alpha}}$
with $c_{\boldsymbol{\alpha}}\ne0$. If $\textbf{x}^{\boldsymbol{\alpha}}$ is this monomial,
define its {\it leading term} by
$\operatorname{LT}(f)=c_{\boldsymbol{\alpha}}\textbf{x}^{\boldsymbol{\alpha}}$.
The partition determined by the leading monomial of $f$ is called
the {\it leading partition} of $f$. In what follows we will construct linearly independent kernel vectors.

(1) When even $\ell\in [2, a]$, define $f_\ell\in V_{2,\ell}$ by
\begin{equation}\label{eq:even-vector}
 f_\ell=\sum_{i=0}^{\ell}(-1)^i\binom{\ell}{i}x_i x_{\ell-i}.
\end{equation}
It is obvious that the leading term of $f_\ell$ is $2x_0x_\ell$. We also have $\mathcal D f_\ell=0$, which follows from
\[
 \begin{aligned}
 \mathcal D f_\ell
 &={}\sum_{i=1}^{\ell}(-1)^i i\binom{\ell}{i}
       x_{i-1}x_{\ell-i}+\sum_{i=0}^{\ell-1}(-1)^i(\ell-i)\binom{\ell}{i}
       x_i x_{\ell-i-1}\\
 &={}\sum_{i=0}^{\ell-1}(-1)^i
  \left((\ell-i)\binom{\ell}{i}-(i+1)\binom{\ell}{i+1}\right)
  x_i x_{\ell-i-1}\\
  &=0.
 \end{aligned}
\]

(2) When odd $\ell\in [3, a]$, define $g_\ell\in V_{3,\ell}$ by
\begin{equation}\label{eq:odd-vector}
 g_\ell=x_0\mathcal E f_{\ell-1}
         -(2a-2\ell+2)x_1f_{\ell-1}.
\end{equation}
The only term involving $x_\ell$ in the right of (\ref{eq:odd-vector}) comes from applying
$\mathcal E$ to $x_{\ell-1}$ in the leading term
$2x_0x_{\ell-1}$ of $f_{\ell-1}$.
Therefore the leading term of $g_\ell$ is
$2(a-\ell+1)x_0^2x_\ell$ with positive coefficient. We claim $\mathcal D g_\ell=0$. In fact, in terms of $\mathcal D x_0=0$,
$\mathcal D x_1=x_0$ and $\mathcal D f_{\ell-1}=0$, computation gives
\[
 \begin{aligned}
 \mathcal D g_\ell
 &=\mathcal D(x_0\mathcal E f_{\ell-1})
   -(2a-2\ell+2)\mathcal D(x_1f_{\ell-1})\\
 &=\mathcal D(x_0)\mathcal E f_{\ell-1}+x_0\mathcal D\mathcal E f_{\ell-1}
   -(2a-2\ell+2)(\mathcal D(x_1)f_{\ell-1}+x_1\mathcal D f_{\ell-1})\\
    &=x_0\mathcal D\mathcal E f_{\ell-1}
   -(2a-2\ell+2)x_0f_{\ell-1}\\
     &=x_0(\mathcal E \mathcal D +\mathcal H)f_{\ell-1}
   -(2a-2\ell+2)x_0f_{\ell-1}\\
    &=x_0\mathcal Hf_{\ell-1}
   -(2a-2\ell+2)x_0f_{\ell-1}\\
 &=(2a-2\ell+2)x_0f_{\ell-1}
   -(2a-2\ell+2)x_0f_{\ell-1}\\
 &=0.
 \end{aligned}
\]

(3) For a partition $\lambda=(\lambda_1,\ldots,\lambda_t)$ with $\lambda_j\in \{2,\ldots,a\}$,
define
\begin{equation}\label{def+C+labta}
 C(\lambda)
 =2\bigl|\{1\le j\le\ell(\lambda):\lambda_j\text{ is even}\}\bigr|
  +3\bigl|\{1\le j\le\ell(\lambda):\lambda_j\text{ is odd}\}\bigr|.
\end{equation}

Fix an integer $b\ge0$. For each partition $\lambda$ with parts in
$\{2,\ldots,a\}$, if $C(\lambda)\le b$, then we define a polynomial $K_\lambda$ to be
\begin{equation}\label{eq:generl-vector}
 K_\lambda=
x_0^{b-C(\lambda)}
  \prod_{\substack{1\le j\le\ell(\lambda)\\\lambda_j\text{ even}}}f_{\lambda_j}
  \prod_{\substack{1\le j\le\ell(\lambda)\\\lambda_j\text{ odd}}}g_{\lambda_j}.
\end{equation}

Obviously, $K_\lambda\in V_{b,\wt(\lambda)}$. We also have
$\mathcal D K_\lambda=0$ in terms of $\mathcal D x_0=0$, $\mathcal D f_\ell=0$ for even $\ell$,
and $\mathcal D g_\ell=0$ for odd $\ell$. So
$$K_\lambda\in V_{b,\wt(\lambda)}\bigcap\ker\mathcal D.$$
If $t=\ell(\lambda)$, then its leading monomial is
\begin{equation}
  x_0^{b-t}\prod_{j=1}^{t}x_{\lambda_j}
\end{equation}
with coefficient
\begin{equation}
 2^t\prod_{\substack{1\le j\le t\\\lambda_j\text{ odd}}}
(a-\lambda_j+1)>0.
\end{equation}
Thus distinct partitions give distinct leading monomials and hence
linearly independent vectors: in any nonzero linear combination,
the largest leading monomial among the participating vectors cannot
be canceled. In particular, $K_\lambda$ has leading partition $\lambda$.
For the empty partition, this construction gives $K_{\varnothing}=x_0^b$.

When $a\ge3$, we also use the polynomial
\begin{equation}\label{eq:delta-vector}
 \Delta=x_0^2x_3^2-6x_0x_1x_2x_3
         +4x_0x_2^3+4x_1^3x_3-3x_1^2x_2^2
 \in V_{4,6}.
\end{equation}
Differentiating its five terms and collecting like monomials gives
\[
 \begin{aligned}
 \mathcal D\Delta={}&(6-6)x_0^2x_2x_3
 +(-12+12)x_0x_1^2x_3\\
 &+(-18+24-6)x_0x_1x_2^2
 +(12-12)x_1^3x_2=0.
 \end{aligned}
\]
Its leading monomial is $x_0^2x_3^2$ with coefficient one.

\section{An algebraic reduction by a M\"obius correction}
\label{sec:reduction}

In this section, we will construct a M\"obius correction $R(q)$ and use it to reduce $G_{n,m}(q)\ge_q0$ in Theorem \ref{thm:strong} to three coefficient nonnegativity inequalities in Proposition \ref{prop:reduction}.

\subsection{Construction of the M\"obius correction }
In this subsection, using the M\"obius inversion, we will construct a M\"obius correction $R(q)$ in terms of Gaussian polynomials.

Recall $P(q)=\qbinom{n-1}{m-1}$. For $\gcd(n,m)>1$, we will seek a polynomial
$R(q)\in\Z[q]$ such that
\[
 \qint m\mid \left(P(q)-\qint{n+1}R(q)\right).
\]
Since $\qint m$ is monic and its simple zeros are precisely the
nontrivial $m$th roots of unity, this is equivalent to
\[
 P(\zeta)=\left.\qint{n+1}\right|_{q\rightarrow\zeta}R(\zeta)
 \qquad(\zeta^m=1,\ \zeta\ne1).
\]
The following lemma provides the required Gaussian evaluation, which will be used to our M\"obius correction.

\begin{lem}\label{lem:root-evaluation}
Let $A,B$ and $r$ be positive integers satisfying $A\ge B$, $r>1$ and $r|B$.
If $\zeta\in\mathbb C$ be a primitive $r$th root of unity, then
\[
  \left.\qbinom{A-1}{B-1}\right|_{q\rightarrow\zeta}
  =\begin{cases}
     0,&r\nmid A,\\[2pt]
     \displaystyle\binom{A/r-1}{B/r-1},&r\mid A.
   \end{cases}
\]
\end{lem}

\begin{proof}
By (\ref{def+Gauss+poly}) of Gaussian polynomials, we have
\begin{equation}\label{gass:A-1+B-1}
\qbinom{A-1}{B-1}
  =\frac{N(q)}
         {D(q)},
\end{equation}
where
\[
N(q):=\prod_{j=A-B+1}^{A-1}(1-q^j),
\quad
D(q):=\prod_{j=1}^{B-1}(1-q^j).
\]
We will evaluate this quotient (\ref{gass:A-1+B-1}) by separating the factors that vanish at $\zeta$.
Since \(\zeta\) is a primitive \(r\)-th root of unity, we have
\[
\zeta^j=1
\iff
r\mid j.
\]
Thus a factor \(1-q^j\) vanishes at \(q=\zeta\) exactly when \(r\mid j\), and every such zero is simple because
\[
\frac{d}{dq}(1-q^j)\Big|_{q\rightarrow\zeta}
=
-j\zeta^{j-1}\neq 0.
\]

For a set, let
$
Z(I):=|\{j\in I:\ r\mid j\}|
$
denote the number of multiples of \(r\) in \(I\).
We also let $B=k\,r$ since \(r\mid B\). For the denominator $D(q)$, we have
\[
Z([1,B-1])=k-1.
\]
 Hence \(D(q)\) has a zero \(q=\zeta\) of order \(k-1\). For the numerator $N(q)$, we have
\begin{equation}\label{number+zeros:A-1+B-1}
Z([A-B+1,A-1])
=
\left\lfloor \frac{A-1}{r}\right\rfloor
-
\left\lfloor \frac{A-B}{r}\right\rfloor
=
\left\lfloor \frac{A-1}{r}\right\rfloor
-
\left\lfloor \frac{A}{r}\right\rfloor
+
k.
\end{equation}

We now split into the following two cases to consider:

\medskip
\noindent
\textbf{Case 1: \(r\nmid A\).}
Then
\[
\left\lfloor \frac{A-1}{r}\right\rfloor
=
\left\lfloor \frac{A}{r}\right\rfloor,
\]
so
\[
Z([A-B+1,A-1])=k.
\]
Thus  $N(q)$ has one more zero \(q=\zeta\) than $D(q)$. Hence
we obtain
\begin{equation}
\left.\qbinom{A-1}{B-1}\right|_{q\rightarrow\zeta}=\left.\frac{N(q)}{D(q)}\right|_{q\rightarrow\zeta}=0.
\end{equation}

\medskip
\noindent
\textbf{Case 2: \(r\mid A\).} Let $A=\ell\, r$. Obviously, $\ell\geq k\geq1$ since $A\geq B$.
Then
\[
\left\lfloor \frac{A-1}{r}\right\rfloor
=
\ell-1,
\quad
\left\lfloor \frac{A}{r}\right\rfloor
=
\ell.
\]
Therefore,  (\ref{number+zeros:A-1+B-1}) gives
\[
Z([A-B+1,A-1])
=
\ell-1-\ell+k
=
k-1.
\]
So the numerator $N(q)$ and denominator $D(q)$ have the zero \(q=\zeta\) of the same order $k-1$.

The nonvanishing factors in $N(q)$ are those with \(r\nmid j\). In the interval
\[
A-B+1\le j\le A-1,
\]
there are \(B-1=k\,r-1\) indices, among which \(k-1\) are multiples of \(r\). Hence each nonzero residue class modulo \(r\) occurs exactly \(k\) times. For the
denominator, it is the same as the numerator. Therefore the products of the nonvanishing factors in \(N(q)\) and \(D(q)\) have the same value at \(q=\zeta\), and their ratio is \(1\).

It remains to compare the vanishing factors. The multiples of \(r\) in $N(q)$ correspond to
\[
j=r\, m,
\quad
m=l-k+1,l-k+2,\dots,l-1.
\]
The multiples of \(r\) in $D(q)$ correspond to
$
j=r\, m,\,
m=1,2,\dots,k-1.
$
Thus, we obtain
\begin{eqnarray*}
\left.\frac{N(q)}{D(q)}\right|_{q\rightarrow\zeta}
&=&\lim_{q\to \zeta}
\frac{
\prod_{m=l-k+1}^{l-1}(1-q^{rm})
}{
\prod_{m=1}^{k-1}(1-q^{rm})
}\\
&
=&\lim_{q\to \zeta}
\frac{
\prod_{m=l-k+1}^{l-1}(1-q^r)\left(1+q^r+q^{2r}+\cdots+q^{(m-1)r}\right)
}{
\prod_{m=1}^{k-1}(1-q^r)\left(1+q^r+q^{2r}+\cdots+q^{(m-1)r}\right)
}\\
&=&\frac{
\prod_{m=l-k+1}^{l-1} m
}{
\prod_{m=1}^{k-1} m
}\\
&
=&\binom{l-1}{k-1}.
\end{eqnarray*}

Substituting \(l=A/r\) and \(k=B/r\), we get
\[
\left.
\begin{bmatrix}
A-1\\
B-1
\end{bmatrix}\right|_{q\rightarrow \zeta}
=
\binom{A/r-1}{B/r-1}.
\]

Combining both two cases yields the proof of this lemma.
\end{proof}

For the remainder of this section, let $n,m$ be integers satisfying
\begin{equation}\label{eq:noncoprime-range}
  3\le m\le n/2,\qquad r=\gcd(n,m)>1.
\end{equation}
If $\zeta$ is a primitive $r$th root of unity, then
$\left.\qint m\right|_{q=\zeta}=0$ and Lemma~\ref{lem:root-evaluation} gives
\[
 P(\zeta)=\binom{n/r-1}{m/r-1}>0.
\]
Thus $\qint m\nmid P(q)$.

Recall the M\"obius function
$\mu:\Z_{\ge1}\to\{-1,0,1\}$ be
\[
 \mu(d)=\begin{cases}
 1,&d=1,\\
 (-1)^k,&d=p_1\cdots p_k\text{ for distinct primes }p_1,\ldots,p_k,\\
 0,&p^2\mid d\text{ for some prime }p.
 \end{cases}
\]

Now we define the M\"obius correction as follows.

\begin{de}
For each $d\mid r$ and $d>1$, let
\begin{eqnarray}
 &R_d(q)&:=\left.\begin{bmatrix}n/d-1\\ m/d-1\end{bmatrix}\right|_{q\rightarrow q^d},\label{def+Rd}\\
  &R(q)&:=-\sum_{\substack{d>1,\,d\mid r}}\mu(d)R_d(q),\quad S(q):=-\sum_{\substack{d>1 ,d\mid r}}\mu(d)q^{L_d}R_d(q),\label{def+R}\\
   &Q(q)&:=\frac{P(q)-\qint{n+1}R(q)}{\qint m},\quad W(q):=\frac{R(q)-S(q)}{1-q^m},\quad\quad
\end{eqnarray}
\end{de}
Obviously, $\deg R_d(q)=(m/d-1)(n-m)$ and $\deg P(q)=(n-m)(m-1)$. For brevity, we also let
\begin{eqnarray}
 && L_d:=\deg P(q)-n-\deg R_d(q)\quad\text{for}\, d\mid r \,\text{and}\, d>1,\\
 &&\delta:=\deg P(q)=(n-m)(m-1).
\end{eqnarray}

Since $3\leq m\leq n/2$, $d\geq2$ and $d\mid m$, a direct computation for $L_d$ gives
\begin{equation}\label{Ld+posi}
L_d
=
(m-m/d-1)(n-m)-m
\ge
m(m-m/d-2)
\ge 0,
\end{equation}
where equality holds if and only if $d=2$ and $m=4$.

In terms of the symmetry of Gaussian polynomials and $L_d+\deg R_d=\delta-n$, it is easy to obtain the reciprocity
\begin{eqnarray}
 R_d(q)&=&q^{\deg R_d(q)}R_d(q^{-1})
 \qquad(d\mid r,\ d>1),\label{symmery+Rd}\\
  S(q)&=&q^{\delta-n}R(q^{-1}).\label{symmery+R}
\end{eqnarray}

The next lemma establishes the polynomiality of $Q(q)$.

\begin{lem}\label{lem:polynomiality}
For each $d\mid r$ and $d>1$, then we have
\begin{itemize}
\item [\rm (i)]$Q(q)\in\Z[q]$ and $\deg Q(q)\leq\delta-m+1$;
\item [\rm (ii)]
$W(q)\in\Z[q]$ and  $q^{\delta-n-m}W(q^{-1})=W(q).$
\end{itemize}
\end{lem}
\begin{proof}
(i)
Set
$
N(q)=P(q)-[n+1]_q R(q).
$
We first show that \(N(q)\) vanishes at every root of \([m]_q\). Let \(\zeta\) be a root of \([m]_q\). Then \(\zeta\) is a primitive \(c\)-th root of unity for some integer \(c>1\) with \(c\mid m\).

We first compute $P(\zeta)$.
Since \[
P(q)=\begin{bmatrix}n-1\\ m-1\end{bmatrix},
\] applying Lemma~3.1 to $P(q)$ gives

\begin{equation}\label{eq:P+zata}
  P(\zeta)=
  \begin{cases}
    0,&c\nmid n,\\[2pt]
    \displaystyle \binom{n/c-1}{m/c-1},&c\mid n.
  \end{cases}
\end{equation}

Next, for each \(d\mid r\) and \(d>1\), consider
$R_d(\zeta).$
It is easy to obtain that the order of \(\zeta^d\) is
\[
c'=\frac{c}{\gcd(c,d)}.
\]
Since \(c\mid m\), \(d\mid r\) and \(r\mid m\), we have
\[
\left.c'\right|\frac{m}{d}.
\]
Applying Lemma~3.1 to \(R_d(\zeta)\)  yields
\begin{equation}\label{eq:Rd+zata+1}
  R_d(\zeta)=
  \begin{cases}
    0,&c'\nmid \frac nd,\\[2pt]
    \displaystyle \binom{\frac{n/d}{c'}-1}{\frac{m/d}{c'}-1}, & c'\mid \frac nd.
  \end{cases}
\end{equation}
Obviously, the condition \(c'\mid n/d\) is equivalent to
\[
\left.\frac{c}{\gcd(c,d)}\right| \frac nd
\iff
\operatorname{lcm}(c,d)\mid n.
\]
By \(d\mid r\) and \(r\mid m\), this is further equivalent to \(c\mid n\). Hence, we derive
\begin{equation}\label{eq:Rd+zata+2}
  R_d(\zeta)=
  \begin{cases}
    0,&c\nmid n,\\[2pt]
    \displaystyle \binom{\frac{n}{\operatorname{lcm}(c,d)}-1}{\frac{m}{\operatorname{lcm}(c,d)}-1}, & c\mid n.
  \end{cases}
\end{equation}
In particular, when \(c\mid n\), the value \(R_d(\zeta)\) depends only on \(\operatorname{lcm}(c,d)\).

We begin to compute $R(\zeta)$ and claim $R(\zeta)=P(\zeta)$. We arrange the proof of this claim as follows.
Now extend the definition by setting \(R_1(q)=P(q)\) and \(\mu(1)=1\). We claim that
\begin{equation}\label{eq+mobi+sum}
\sum_{d\mid r}\mu(d)R_d(\zeta)=0.
\end{equation}
If \(c\nmid n\), then (\ref{eq+mobi+sum}) is immediate from (\ref{eq:P+zata}) and (\ref{eq:Rd+zata+2}). So in the following, we suppose \(c\mid n\). Choose a prime \(p\mid c\). Since \(c\mid m\) and \(c\mid n\), we have \(p\mid m\) and \(p\mid n\), hence \(p\mid r\). Note $\mu(d)=0$ for $d$ having a square divisor. So we don't need to consider such case in (\ref{eq+mobi+sum}). For every squarefree divisor \(d\mid r\) with \(p\nmid d\), pair \(d\) with \(d\,p\). Then \(d\,p\mid r\) and
$
\mu(dp)=-\mu(d).
$
Moreover, it follows from \(p\mid c\) that
\[
\operatorname{lcm}(c,dp)=\operatorname{lcm}(c,d),
\]
which combining (\ref{eq:Rd+zata+2}) gives \(R_{dp}(\zeta)=R_d(\zeta)\). Therefore
\[
\mu(d)R_d(\zeta)+\mu(dp)R_{dp}(\zeta)
=
\mu(d)R_d(\zeta)-\mu(d)R_d(\zeta)=0.
\]
In consequence, all terms cancel in pairs, which implies (\ref{eq+mobi+sum}).

From (\ref{eq+mobi+sum}) and the definition of \(R(q)\), we obtain
\[
P(\zeta)-R(\zeta)
=
P(\zeta)+\sum_{\substack{d>1,\, d\mid r}}\mu(d)R_d(\zeta)
=
\sum_{d\mid r}\mu(d)R_d(\zeta)=0.
\]
Hence
\begin{equation}\label{eq+R+P+zata}
R(\zeta)=P(\zeta).
\end{equation}
In what follows we will compute $N(\zeta)$ by considering two cases:

\textbf{Case 1}: \(c\mid n\). Then \(n+1\equiv 1\pmod c\), so
\[
[n+1]_{\zeta}
=
\frac{1-\zeta^{n+1}}{1-\zeta}
=
\frac{1-\zeta}{1-\zeta}=1.
\]
Together with (\ref{eq+R+P+zata}), this yields
\[
N(\zeta)=P(\zeta)-[n+1]_{\zeta}R(\zeta)=P(\zeta)-R(\zeta)=0.
\]

\textbf{Case 2}: \(c\nmid n\). Combining (\ref{eq:P+zata}) and (\ref{eq:Rd+zata+2}) gives \(P(\zeta)=0\) and \(R(\zeta)=0\), so again
\[
N(\zeta)=P(\zeta)-[n+1]_{\zeta}R(\zeta)=0.
\]

Therefore \(N(q)=P(q)-[n+1]_q R(q)\) vanishes at every root of \([m]_q\). Because each root of \([m]_q\) is simple, $Q(q)$ is a polynomial in $q$.
Note for any two polynomials $U(q)$ and $V(q)$ in \(\mathbb{Z}[q]\) that the quotient $U(q)/V(q)$ is in \(\mathbb{Z}[q]\) when $V(q)|U(q)$ and the leading coefficient of $V(q)$ is $1$. In consequence, $Q(q)\in\mathbb{Z}[q].$

Finally, it remains to bound the degree $\deg Q(q)\leq\delta-m+1$. In terms of $Q(q)=N(q)/{[m]_q}$, it suffices to prove $\deg N(q)\leq \delta$.
It follows from $L_d=\delta-n-\deg R_d(q)$ and the nonnegativity of $L_d$ in (\ref{Ld+posi}) that
$
\deg R(q)+n\le \delta
$
for every \(d>1\).
Therefore
$
\deg\bigl([n+1]_q R(q)\bigr)
\le
n+\deg R(q)
\le
\delta.
$
In consequence, we have
\[
\deg N(q)\le \delta.
\]

(ii)
For each $d\mid r$ and $d>1$, let
\begin{eqnarray}\label{def+Wd}
 W_d(q):=\frac{R_d(q)(1-q^{L_d})}{1-q^m}.
\end{eqnarray}
Combining (\ref{def+R}) and (\ref{def+Wd}) immediately implies
\begin{eqnarray}\label{def+Wd+W}
W(q)=\frac{R(q)-S(q)}{1-q^m}
=-\sum_{\substack{d\mid r, d>1}}\mu(d)
\frac{R_d(q)(1-q^{L_d})}{1-q^m}=-\sum_{\substack{d\mid r, d>1 }}\mu(d)W_d(q).
\end{eqnarray}
Obviously, $q^{\delta-n-m}W_d(q^{-1})=W_d(q)$, which by (\ref{def+Wd+W}) immediately gives $q^{\delta-n-m}W(q^{-1})=W(q).$
In the following, in order to prove $W(q)\in\mathbb{Z}[q]$, it suffices to show $W_d(q)\in\mathbb{Z}[q]$.
We can prove this in a similar way with the proof $Q(q)\in\mathbb{Z}[q]$ in terms of the following two cases:

\textbf{Case 1: \(c\nmid n\).}
By Lemma~3.1, we have \(R_d(\zeta)=0\), which by (\ref{def+Wd}) implies $W_d(\zeta)=0$.

\textbf{Case 2: \(c\mid n\).}
Again by Lemma~3.1, \(R_d(\zeta)\neq 0\). We claim that \(c\mid L_d\). Recall that
\[
L_d=\delta-n-\deg R_d,
\]
where
$
\delta=(n-m)(m-1),\deg R_d=\left(\frac md-1\right)(n-m).
$
Thus, we obtain
\[
\delta-\deg R_d
=
(m-1)\bigl(n-m\bigr)-\left(\frac md-1\right)(n-m)
=
\frac{m}{d}(d-1)(n-m).
\]
Obviously, $m/d\in \mathbb{Z}$ by $d|r$ and $r|m$, and \(c\mid n-m\) by \(c\mid m\). Hence, we obtain
\[
L_d=\delta-n-\deg R_d\equiv \delta-\deg R_d\equiv 0\pmod c.
\]
Therefore
$
1-\zeta^{L_d}=0,
$  which implies $W_d(\zeta)=0$.

In consequence, in both cases the numerator \(R_d(q)(1-q^{L_d})\) vanishes at every nontrivial \(m\)-th root of unity. Obviously, it is also true for $q=1$.
So similar to the proof of $Q(q)\in\mathbb Z[q]$, we derive $W_d(q)\in\mathbb Z[q]$.

This completes the proof.
\end{proof}

\subsection{Reduction of coefficient positivity of $G_{n,m}(q)$}

The following result says that three nonnegativity inequalities imply that of coefficients of $G_{n,m}(q)$.
\begin{prop}\label{prop:reduction}
Let $C_k:=[q^k](1-q)P(q)$ and $R_k:=[q^k]R(q)$
for $k\geq0$. If
\[
  R(q)\ge_q0,\quad W(q)\ge_q0\quad and \quad C_k\ge R_k\quad(0\le k\le\delta/2),
\]
then we have $G_{n,m}(q)\ge_q0$.
\end{prop}

\begin{proof}
Recall (\ref{eq:G-to-R+Q}) that
\begin{eqnarray}\label{def+Gnm}
  G_{n,m}(q)  =\frac{R(q)}{1-q^m}+\frac{Q(q)}{1-q^{n+1}}.
\end{eqnarray}
Because $\frac{1}{1-q^m}\ge_q0$, the condition $R(q)\ge_q0$ implies $\frac{R(q)}{1-q^m}\ge_q0$. So by (\ref{def+Gnm}), in order to prove
$G_{n,m}(q)\ge_q0$, it suffices to show $Q(q)\ge_q0$.

In terms of the definition of $Q(q)$, we can rewrite
\begin{eqnarray}\label{expres+Q}
  Q(q)=\frac{(1-q)P(q)}{1-q^m}-\frac{(1-q^{n+1})R(q)}{1-q^m}.
\end{eqnarray}
Let $Q(q)=\sum_{k\geq0}Q_kq^k$. Expanding the denominator in (\ref{expres+Q}) as a formal power series gives
\begin{eqnarray}\label{expres+Qk}
 Q_k=\sum_{j\ge0}
  \bigl(C_{k-jm}-R_{k-jm}+R_{k-n-1-jm}\bigr)
\end{eqnarray}
for $k\geq0$. By Lemma \ref{lem:polynomiality} (i), we have $\deg Q\le\delta-m+1$.
We will prove $Q_k\geq0$ for $0\le k\le\delta-m+1$ in terms of the following two cases:

\textbf{Case 1}: $0\le k\le(\delta-m+1)/2$. We have $k\leq \delta/2$. It follows from the condition $C_k\ge r_k$ for $0\le k\le\delta/2$ and $R_k\geq0$ for $k\geq0$ that
\begin{eqnarray}
 Q_k=\sum_{j\ge0}
  \left(C_{k-jm}-R_{k-jm}\right)+\sum_{j\ge0}R_{k-n-1-jm}\geq0.
\end{eqnarray}

\textbf{Case 2}: $(\delta-m+1)/2\leq k\le \delta-m+1$.
Using
$q^\delta P(q^{-1})=P(q)$,
  $[n+1]_{q^{-1}}=q^{-n}[n+1]_q$ and
  $[m]_{q^{-1}}=q^{1-m}[m]_q$,
we obtain
\begin{eqnarray}
  q^{\delta-m+1}Q(q^{-1})
  &=&\frac{P(q)-\qint{n+1}q^{\delta-n}R(q^{-1})}{\qint m}\\
  &=&\frac{P(q)-\qint{n+1}S(q)}{\qint m},\\
   &=&\frac{(1-q)P(q)}{1-q^m}-\frac{R(q)}{1-q^m}
  +W(q)+\frac{q^{n+1}S(q)}{1-q^m},\label{eq:reduction-Q-reflection}
\end{eqnarray}
which is a polynomial because
$\deg Q\le\delta-m+1$. Let $S(q)=\sum_kS_kq^k$, whose coefficients are
those of $R(q)$ in reverse order by (\ref{symmery+R}). So $R_k\geq0$ implies $S_k\geq0$ for $k\geq0$.

Taking coefficients of $q^k$ in (\ref{eq:reduction-Q-reflection}) yields
\begin{eqnarray}
 Q_{\delta-m+1-k}
 &=&\sum_{j\ge0}(C_{k-jm}-R_{k-jm})+[q^k]W(q)+\sum_{j\ge0}S_{k-n-1-jm},
\end{eqnarray}
which is nonnegative for $0\le k\le(\delta-m+1)/2$ in terms of $C_k\ge R_k$, $W(q)\ge_q0$ and
$S(q)\ge_q0$.
Let $k'=\delta-m+1-k$. We immediately have $(\delta-m+1)/2\leq k'\le \delta-m+1$ for $0\le k\le(\delta-m+1)/2$.
In consequence, we obtain $ Q_{k'}\geq0$ for $(\delta-m+1)/2\leq k'\le \delta-m+1$.

Combining this two cases gives $Q_k\geq0$ for $0\le k\le\delta-m+1$. Hence $G_{n,m}(q)\ge_q0$. This completes the proof.

\end{proof}

In the following sections, we will prove three nonnegativity inequalities in Proposition \ref{prop:reduction}.

\section{Partition transformation and positivity of the M\"obius correction $R(q)$}
\label{sec:correction-positive}

The aim of this section is to prove the M\"obius correction $R(q)\geq_q 0$. 
In order to complete this, we will turn the formula (\ref{def+Rd}) of $R(q)$ to a new one (\ref{def+Rq+new}) in terms of partition. In addition, we define a splitting partition transformation (\ref{eq:partition-split}). With the help of the new formula and the partition transformation, we show a injection, which implies $R(q)\geq_q 0$ in Propositon \ref{prop:R-positive}.

Fix $n,m\in\Z$ with $3\le m\le n/2$ and $r=\gcd(n,m)>1$, as in Section~\ref{sec:reduction}. Let $\lambda$
range over partitions of nonnegative integers, and define
\[
 \mathcal P:=\{\varnothing\}\cup
 \{\lambda\ne\varnothing:
        \ell(\lambda)\le m-1,\ \lambda_1\le n-m\}.
\]
For each $d\mid r$ with $d>1$, define
\begin{eqnarray}\label{def+SS}
  \mathcal S_d:=\{\lambda\in\mathcal P:
       \ell(\lambda)\le m/d-1,\quad
       d\mid\lambda_i\quad {\text{for}}\quad 1\le i\le\ell(\lambda)\}.
\end{eqnarray}
Since $d\mid n-m$, the map $\lambda\mapsto(\lambda_1/d,\ldots,
\lambda_{\ell(\lambda)}/d)$ is a bijection from $\mathcal S_d$
onto the partitions with at most $m/d-1$ parts, each at most
$(n-m)/d$. In consequence, in terms of (\ref{def+Rd}), we have
\begin{eqnarray}
  R_d(q)=\sum_{\lambda\in\mathcal S_d}q^{\wt(\lambda)}.
\end{eqnarray}
For a set $I$, recall its characteristic function
\[
 \textbf{1}_I=\begin{cases}
 1,&x\in I;\\
 0,&otherwise.
 \end{cases}
\]
Define $M:\mathcal P\to\Z$ by
\begin{eqnarray}\label{def+M+labta}
M(\lambda):=-\sum_{\substack{d\mid r,d>1}}
                     \mu(d)\chi_{\mathcal S_d}(\lambda).
\end{eqnarray}

Then the correction (\ref{def+R}) for $R(q)$ becomes to
\begin{eqnarray}\label{def+Rq+new}
  R(q)=-\sum_{\substack{d>1,\,d\mid r}}\mu(d)\sum_{\lambda\in\mathcal S_d}q^{\wt(\lambda)}=\sum_{\lambda\in\mathcal P}M(\lambda)q^{\wt(\lambda)}.
\end{eqnarray}

For $n\in\Z$ and $n\geq1$, let $\rad (n)$ denote the product of the distinct
prime divisors of $n$ and $\rad(1)=1$. For a partition $\lambda$, define

\begin{equation}\label{def:H}
  H(\lambda):=
 \begin{cases}
 \rad\bigl(\gcd(r,\lambda_1,\ldots,
                    \lambda_{\ell(\lambda)})\bigr),
     &\lambda\ne\varnothing,\\
 \rad (r),&\lambda=\varnothing.
 \end{cases}
\end{equation}

For squarefree $d\mid r$ with $d>1$, it is easy to prove
\[
 d\mid H(\lambda)\ \text{and}\
 d(\ell(\lambda)+1)\le m
 \quad\Longrightarrow\quad
 \lambda\in\mathcal S_d.
\]
Conversely, if $\lambda\in\mathcal S_d$, then
$d(\ell(\lambda)+1)\le m$ and $d\mid\lambda_i$ for every $i$.
The squarefreeness of $d$ therefore gives $d\mid H(\lambda)$.
Since $\mu(d)=0$ for nonsquarefree $d$ and
$\ell(\lambda)+1\le m$,
\begin{equation}\label{eq:partition-multiplicity}
 M(\lambda)
 =-\sum_{\substack{d\mid H(\lambda)\\
                    d>1,\ d(\ell(\lambda)+1)\le m}}\mu(d)
 =1-\sum_{\substack{d\mid H(\lambda)\\
                    d(\ell(\lambda)+1)\le m}}\mu(d).
\end{equation}
For $\lambda=\varnothing$, the identity
$\sum_{d\mid r}\mu(d)=0$ gives
\[
 M(\varnothing)=-\sum_{\substack{d\mid r,d>1}}\mu(d)=1.
\]

The next lemma identifies which prime divisors of $H(\lambda)$ can
contribute to $M(\lambda)$ and bounds any negative multiplicity.

\begin{samepage}
\begin{lem}
\label{lem:multiplicity}
Let $\lambda\in\mathcal P$ be nonempty and $s$ be the
number of distinct prime divisors of $H(\lambda)$. We have
\begin{itemize}
\item [\rm (i)]If $M(\lambda)\ne0$, then $H(\lambda)>1$ and its smallest
 prime divisor $p$ satisfies $\ell(\lambda)\le m/p-1$;
\item [\rm (ii)]If $M(\lambda)\ge2$ or $M(\lambda)<0$, then $H(\lambda)$
 has a second-smallest
 prime divisor $\phi$, and $\ell(\lambda)\le m/\phi-1$;
\item [\rm (iii)] If $s\le3$, then $M(\lambda)\ge0$;
\item [\rm (iv)]
If $s\ge1$, then $M(\lambda)\le2^{s-1}$. If $M(\lambda)=-c<0$,
 then $c\le2^{s-1}-1$.
\end{itemize}
\end{lem}
\end{samepage}

\begin{proof}
It is easy to prove (i), (ii) and (iii). For (iv), a squarefree integer with \(s\) prime factors has
\(2^{s-1}\) divisors of each M\"obius sign. There are
\(2^{s-1}-1\) nontrivial divisors with sign \(+1\). Counting
the possible positive and negative terms in \eqref{eq:partition-multiplicity}
proves the last assertion.
\end{proof}

For a nonempty partition $\lambda=(\lambda_1,\lambda_2,\ldots,
                    \lambda_{\ell(\lambda)})$ and $h\in\Z$ with
$1\le h<\lambda_1$, define a splitting partition transformation
\begin{equation}\label{eq:partition-split}
 T_h(\lambda):=
 (h,\lambda_1-h,\lambda_2,\ldots,
                    \lambda_{\ell(\lambda)})^\downarrow,
\end{equation}
where the superscript $\downarrow$ denotes rearrangement in weakly
decreasing order. The sequence
$\lambda_2,\ldots,\lambda_{\ell(\lambda)}$ is empty when
$\ell(\lambda)=1$.
Obviously,
\[
 \wt(T_h(\lambda))=\wt(\lambda),\quad
 \ell(T_h(\lambda))=\ell(\lambda)+1.
\]
The following lemma shows that a split partition determines its
source and splitting parameter. We will use this to prove
the injectivity in the proof of Proposition~\ref{prop:R-positive}.

\begin{lem}
\label{lem:split}
Let $\lambda$ be a nonempty partition.
Suppose that integers $H$ in (\ref{def:H}) and $h$ satisfy
\[
 \begin{gathered}
 H>1,\quad 0<h<H,\quad H\nmid2h,\quad \lambda_1>2h,\quad H\mid\lambda_i
                  \quad(1\le i\le\ell(\lambda)).
 \end{gathered}
\]
Then $T_h(\lambda)$ uniquely determines the pair $(h,\lambda)$ by an
inverse independent of $H$.
\end{lem}

\begin{proof}
Write $T_h(\lambda)=(\mu_1,\ldots,\mu_{\ell(\lambda)+1})$.
(1) We first determine $h$ from the smallest entry of $T_h(\lambda)$.
In terms of the hypotheses conditions, we obtain
\begin{equation}\label{ineq:labta}
\lambda_1-h>h,\quad \lambda_i\ge H>h\quad(1\le i\le\ell(\lambda)).
\end{equation}

Since $T_h(\lambda)$ is the decreasing rearrangement of
$(h,\lambda_1-h,\lambda_2,\ldots,
\lambda_{\ell(\lambda)})$, it follows from (\ref{ineq:labta}) that
\[
 \begin{aligned}
 \mu_1&\ge\cdots\ge\mu_{\ell(\lambda)}>
                 \mu_{\ell(\lambda)+1}=h,\\
 \mu_i&\ge\min\{H,\lambda_1-h\}>h
                 \quad(1\le i\le\ell(\lambda)).
 \end{aligned}
\]
Thus we obtain $h=\mu_{\ell(\lambda)+1}$. 

(2) In what follows we will determine $\lambda$. To recover $\lambda$, it suffices to delete this
entry $h$ and apply one of the following replacements to the remaining entries before
reordering:
\[
 \begin{cases}
  \mu_1-h\longmapsto\mu_1,
    &\text{if }\mu_j=\mu_1-h
                  \text{ for some }1\le j\le\ell(\lambda),\\
  \mu_1\longmapsto\mu_1+h,&\text{otherwise}.
 \end{cases}
\]

Now we verify this rule according to whether the largest entry $\lambda_1$ of
$\lambda$ occurs more than once:

\textbf{Case 1}: If
$\lambda_j=\lambda_1$ for some $2\le j\le\ell(\lambda)$, then
\[
 \mu_1=\lambda_1,\quad
 \mu_1-h=\lambda_1-h,\quad
 \bigl|\{i:1\le i\le\ell(\lambda),\
                      \mu_i=\mu_1-h\}\bigr|=1.
\]
Indeed, $\lambda_1-h$ occurs among
$\mu_1,\ldots,\mu_{\ell(\lambda)}$, but
$H\nmid(\lambda_1-h)$ whereas $H\mid\lambda_i$
for $2\le i\le\ell(\lambda)$.
So, the first replacement therefore recovers $\lambda$.

\textbf{Case 2}: If $\ell(\lambda)=1$, or if
$\ell(\lambda)\ge2$ and $\lambda_2<\lambda_1$, then
\[
 \lambda_1-\lambda_i\ge H>h
                \quad(2\le i\le\ell(\lambda)),
 \quad \mu_1=\lambda_1-h.
\]
Moreover, $\mu_1-h=\lambda_1-2h$ and $H\nmid(\mu_1-h)$ because
$H\mid\lambda_1$ and $H\nmid2h$. Since
$\{\mu_1,\ldots,\mu_{\ell(\lambda)}\}
 =\{\lambda_1-h,\lambda_2,\ldots,
                     \lambda_{\ell(\lambda)}\}$,
we have
$\mu_1-h\notin\{\mu_1,\ldots,\mu_{\ell(\lambda)}\}$.
Thus the second replacement recovers $\lambda$. 

The rule depends only on $T_h(\lambda)$. This completes the proof.
\end{proof}

Now we present the nonnegativity of coefficients of $R(q)$ as follows.
\begin{prop}
\label{prop:R-positive}
For the polynomial $R(q)$, we have $R(q)\geq_q 0$.
\end{prop}

\begin{proof}
By (\ref{def+Rq+new}), for each $k\ge0$, the coefficient of $q^k$ in $R(q)$ is
\begin{eqnarray*}
 R_k&=&
 \sum_{\substack{\lambda\in\mathcal P\\
                  \wt(\lambda)=k,\ M(\lambda)>0}}M(\lambda)
 - \sum_{\substack{\lambda\in\mathcal P\\
                  \wt(\lambda)=k,\ M(\lambda)<0}}(-M(\lambda))\\
&\geq&\bigl|\{\lambda\in\mathcal P:
       \wt(\lambda)=k,\ M(\lambda)=1\}\bigr|
- \sum_{\substack{\lambda\in\mathcal P\\
                  \wt(\lambda)=k,\ M(\lambda)<0}}(-M(\lambda)).
\end{eqnarray*}
So in order to show $R(q)\geq_q 0$, it suffices to prove 
\begin{eqnarray}\label{ineq+M}
 \bigl|\{\lambda\in\mathcal P:
       \wt(\lambda)=k,\ M(\lambda)=1\}\bigr|
\geq\sum_{\substack{\lambda\in\mathcal P\\
                  \wt(\lambda)=k,\ M(\lambda)<0}}-M(\lambda).
\end{eqnarray}

Fix $\lambda\in\mathcal P$ with $M(\lambda)<0$, and set
$c=-M(\lambda)$ and $H=H(\lambda)$. For each such partition with $M(\lambda)<0$, we will use $T_h(\lambda)$ to construct $c$ distinct
partitions $\lambda'$ with $\wt(\lambda')=\wt(\lambda)$ and $M(\lambda')=1$, with no overlap
between different sources.
Obviously, the partition $\lambda$ is nonempty since $M(\lambda)<0$ and $M(\varnothing)=1$.
Thus, by Lemma~\ref{lem:multiplicity} (iii) and (iv), $H$ has $s\ge4$ distinct prime
divisors and $c\le2^{s-1}-1$. We first prove the next claim.
\begin{cla}
In terms of $T_h(\lambda)$, at least $c$ choices of the splitting parameter
$h$ are available. 
\end{cla}
\begin{proof}Let $p$ be the smallest prime divisor of $H$ and set
\[
 \mathcal A_\lambda:=
 \{h\in\Z_{\ge1}:p\mid h,\ 2h<H,\
                         \gcd(h/p,H/p)=1\}.
\]
In terms of definition of $H$, the integer $H/p$ is odd. Writing $h=pj$, we have
\[|\mathcal A_\lambda|=|\{
 1\le j<\frac{H/p}{2},\quad \gcd(j,H/p)=1.\}|
\]
The units $j$ modulo $H/p$ occur in pairs $j$ and $H/p-j$. Since
$H/p$ is squarefree, this pairing and the product formula for
Euler's totient function give
\[
 |\mathcal A_\lambda|
 =\frac12\prod_{\substack{r\mid H\\r\ne p}}(r-1)
 \ge\frac{2\cdot4\cdot6}{2}\,2^{s-4}
 =3\cdot2^{s-1}>c,
\]
where for the first inequality from the first three primes dividing $H/p$ are at
least $3,5,7$, and every remaining prime contributes a factor of
at least $2$. This proves this claim.
\end{proof}

From this claim, we assume that
$h_1(\lambda)<\cdots<h_c(\lambda)$ are the first $c$ elements of
$\mathcal A_\lambda$. In terms of the definition (\ref{def+M+labta}) of $M(\lambda)$, we next verify that each split
$\nu=T_{h_i(\lambda)}(\lambda)$ belongs to $\mathcal S_p$ and satisfies
$M(\nu)=1$. In order to prove $\nu\in\mathcal S_p$, by (\ref{def+SS}), we need to show that $\ell(\nu)\le m/p-1,\,
       p\mid\nu_i\quad {\text{for}}\quad 1\le i\le\ell(\nu)\}$. Put $h=h_i(\lambda)$. Since $H\mid\lambda_j$ for
$1\le j\le\ell(\lambda)$, we have
\[
 \lambda_1\ge H>2h>0.
\]
Thus $0<h<H$, $\lambda_1>2h$, and $H\nmid2h$, so the conditions of
Lemma~\ref{lem:split} are satisfied. Then by Lemma~\ref{lem:split}, we obtain $\lambda'=\nu=T_{h_i(\lambda)}(\lambda)$. The largest entry of $\nu$ is at most
$\lambda_1\le n-m$. Let $\phi$ be the second-smallest prime divisor
of $H$. By Lemma~\ref{lem:multiplicity} (2),
$\ell(\lambda)+1\le m/\phi$, which combines
$p\phi\mid m$ and $\phi>p$ to give
\[
 \frac mp-\frac m\phi=\frac{m}{p\phi}(\phi-p)\ge1,
 \qquad
 \ell(\nu)=\ell(\lambda)+1\le\frac m\phi\le\frac mp-1.
\]
In consequence, $\nu\in\mathcal P$. To compute $H(\nu)$, observe that
\[
 \gcd(r,\nu_1,\ldots,\nu_{\ell(\lambda)+1})
 =\gcd(r,\lambda_1,\ldots,\lambda_{\ell(\lambda)},h).
\]
Using $\rad(\gcd(u,v))=\gcd(\rad (u),v)$ for $u,v\in\Z_{\ge1}$, we obtain
\[
 \begin{aligned}
 H(\nu)
 &=\rad\bigl(\gcd(r,\lambda_1,\ldots,
                       \lambda_{\ell(\lambda)},h)\bigr)=\gcd(H,h)=p,
 \end{aligned}
\]
where the last equality follows from $h=p(h/p)$,
$\gcd(h/p,H/p)=1$ and the squarefreeness of $H$. Hence, we obtain
$p\mid H(\nu)$ and $\ell(\nu)\le m/p-1$, so
$\nu\in\mathcal S_p$. Since $H(\nu)=p$,
by \eqref{eq:partition-multiplicity}, we also have $M(\nu)=-\mu(p)=1$.
From the above, we have construct $c$ distinct
partitions $\lambda'$ with $\wt(\lambda')=\wt(\lambda)$ and $M(\lambda')=1$.

Finally, we show that the resulting partitions are distinct, even
when they arise from different $\lambda$. For every
$\lambda\in\mathcal P$ with $M(\lambda)<0$ and
$1\le i\le-M(\lambda)$, define
\[
 \Phi(\lambda,i):=T_{h_i(\lambda)}(\lambda).
\]
By Lemma~\ref{lem:split}, $\Phi(\lambda,i)$ determines both
$\lambda$ and $h_i(\lambda)$. The increasing order $h_1(\lambda)<\cdots<h_c(\lambda)$ of
$\mathcal A_\lambda$ then determines $i$. Thus $\Phi$ is
injective and preserves weight. So we obtain for every $k\ge0$ that
\[
 \bigl|\{\nu\in\mathcal P:
       \wt(\nu)=k,\ M(\nu)=1\}\bigr|
 \ge
 \sum_{\substack{\lambda\in\mathcal P\\
                  \wt(\lambda)=k,\ M(\lambda)<0}}(-M(\lambda)),
\]
which implies (\ref{ineq+M}). This completes the proof.
\end{proof}

\section{Positivity of coefficients of $(1-q)P(q)-R(q)$}
\label{sec:domination}

In this section, we will prove $[q^k](1-q)P(q)\geq[q^k]R(q)$
for $0\le k\le\delta/2$, i.e., $C_k\geq R_k$.

Let $n,m\in\Z$ satisfy
$3\le m\le n/2$ and $r=\gcd(n,m)>1$.
Use the spaces $V_{b,k}$ and the derivation $\mathcal D$ of
Section~\ref{sec:kernel} with $a=n-m$ and $b=m-1$. Then
$P(q)=\mathcal G_{n-m,m-1}(q)$, $\delta=(n-m)(m-1)$, and by
Lemma~\ref{lem:kernel} we have
\[
 C_k=\dim_{\Q}\ker\bigl(
 \mathcal D:V_{m-1,k}\longrightarrow V_{m-1,k-1}\bigr)
 \qquad(0\le k\le\delta/2).
\]
We will compare $C_k$ with $R_k$ in terms of $m\ne6$ and $m=6$.

\subsection{Positivity for $m\ne6$ from linearly independent kernel vectors}
\label{sec:five1}

In order to prove for $m\ne6$ that $ C_k\ge R_k$ for $0\le k\le\delta/2$, we need to construct linearly independent kernel vectors in terms of subsection \ref{sec:kernel}
\begin{prop}
\label{prop:domination}
If $m\ne6$, then we have
$ C_k\ge R_k$ for $0\le k\le\delta/2$.
\end{prop}

\begin{proof}
For each $\lambda\in\mathcal P$ with $M(\lambda)>0$ and each
$i\in[1, M(\lambda)]$, we will construct a nonzero polynomial
\[
 v_{\lambda,i}\in
 V_{m-1,\wt(\lambda)}\cap\ker\mathcal D
\]
such that all such assigned polynomials have distinct leading monomials.
The polynomials of any fixed weight will then be linearly independent.
Once the construction is complete, for $0\le k\le\delta/2$, by Lemma~\ref{lem:kernel}, we have

\[
 \begin{aligned}
 C_k
 &=\dim_{\Q}\ker\bigl(
     \mathcal D:V_{m-1,k}\longrightarrow V_{m-1,k-1}\bigr)\\
 &\ge
   \sum_{\substack{\lambda\in\mathcal P\\\wt(\lambda)=k}}
        \max\{M(\lambda),0\}\\
 &\ge
   \sum_{\substack{\lambda\in\mathcal P\\\wt(\lambda)=k}}
        M(\lambda)\\
 & =R_k.
 \end{aligned}
\]
It therefore remains to construct such polynomials and distinguish
their leading monomials.

Suppose $m\ne6$. For each partition $\lambda$, we first assign a polynomial in terms of the following rule: For $\lambda=\varnothing$, define
\[
 v_{\lambda,1}:=x_0^{m-1}.
\]
For a nonempty $\lambda\in\mathcal P$ with $M(\lambda)>0$, put
$H=H(\lambda)$. By Lemma~\ref{lem:multiplicity} (i), the smallest
prime divisor $p$ of $H$ satisfies
$\ell(\lambda)\le m/p-1$. Moreover,
$H\mid\lambda_j$ and $\lambda_j\ge p\ge2$ for every
$j\in[1,\ell(\lambda)]$. Hence, for $C(\lambda)$ defined in (\ref{def+C+labta}), we have
\[
 \begin{cases}
 C(\lambda)=2\ell(\lambda)\le m-2,&p=2,\\
 C(\lambda)\le3\ell(\lambda)
       \le3(m/p-1)\le m-3,&p\ge3.
 \end{cases}
\]
So, $C(\lambda)\leq m-1$, which satisfies the condition of definition of $C(\lambda)$. Hence we can define a polynomial $K_\lambda$ by (\ref{eq:generl-vector}).  Take $v_{\lambda,1}:=K_\lambda$, and we have
$v_{\lambda,1}\in V_{m-1,\wt(\lambda)}\cap\ker\mathcal D$. In addition, its leading monomial is
\[
 x_0^{m-1-\ell(\lambda)}
       \prod_{j=1}^{\ell(\lambda)}x_{\lambda_j},
\]
which is determined by $\lambda$. Therefore, for different $\lambda$, such first polynomials have
distinct leading monomials.

Assume $M(\lambda)\ge2$. In what follows we will assign polynomials to $i\in[2, M(\lambda)]$ additional copies and do it by the following two cases: 

\smallskip\noindent\textbf{Case 1.} $H\ne6$.

We first show that there are enough splitting parameters (at least $M(\lambda)-1$). Let $s$
be the number of distinct prime divisors of $H$, and define
\[
 \mathcal A_\lambda
 :=\{h\in\Z_{\ge2}:2h<H,\ \gcd(h,H)=1\}.
\]

We have the next claim.
\begin{cla}
\[
 |\mathcal A_\lambda| \ge M(\lambda)-1.
\]
\end{cla}
\begin{proof}
By Lemma~\ref{lem:multiplicity}(ii) and (iv),
$s\ge2$ and $M(\lambda)\le2^{s-1}$.
Pairing the reduced residues $h$ and $H-h$ gives
\[
 |\mathcal A_\lambda|
 =\frac12\prod_{\substack{q\mid H\\q\text{ prime}}}(q-1)-1
 \ge2^{s-1}-1
 \ge M(\lambda)-1.
\]
Indeed, if $H$ is odd, then every factor $q\geq3$.
If $H$ is even, then one factor $q=2$, but
$H\ne6$ and $s\ge2$ ensure a prime divisor $q\ge5$ and every remaining $q\geq3$.
\end{proof}
Thus, from this calim, we let $h_2(\lambda)<\cdots<h_{M(\lambda)}(\lambda)$ be the first
$M(\lambda)-1$ elements of $\mathcal A_\lambda$. We now verify that each $h_i(\lambda)$ chosen produces an admissible
polynomial. We claim:
\begin{cla}
Fix $i\in[2, M(\lambda)]$ and let $h=h_i(\lambda)$ and $\nu=T_h(\lambda)$ as in
\eqref{eq:partition-split}. Then there exists a $v_{\lambda,i}
 \in V_{m-1,\wt(\lambda)}\cap\ker\mathcal D.$
\end{cla}
\begin{proof}
Since
\[
 2\le h\,\, {\text{and}}\,\, 2h<H\le\lambda_1\le n-m,
\]
both $h$ and $\lambda_1-h$ lie in $\{2,\ldots,n-m\}$.
The unchanged entries of $\lambda$ lie in this set as well.
Let $\phi$ be the second-smallest prime divisor of
$H$. By Lemma~\ref{lem:multiplicity}(ii), we obtain
$\ell(\lambda)+1\le m/\phi$. We begin to consider the following two cases:

(i) If $H$ is odd, then $\phi\ge5$ and in terms of the definition (\ref{def+C+labta}) of $C(\lambda)$, we have 
\[
 C(\nu)\le3(\ell(\lambda)+1)
       \le\frac{3m}{\phi}\le m-1.
\]

(ii) If $H$ is even, then the entries of $\lambda$ are even and $h$
is odd. Splitting $\lambda_1$ replaces one even entry by two odd
entries, we obtain
\[
 C(\nu)=2(\ell-1)+3\times2=2\ell(\lambda)+4
       \le\frac{2m}{\phi}+2\le m-1,
\]
where the last inequality follows from the fact: if $\phi=3$, then $6\mid m$ and $m\ne6$
give $m\ge12$; if $\phi\ge5$, then $2\phi\mid m$. 
Also, we obtain
\[
 \ell(\nu)=\ell(\lambda)+1\le m/\phi\le m-1,
 \quad \nu_1\le n-m.
\]
Thus $\nu\in\mathcal P$ and $\nu_1\le n-m$ satisfies the condition of the definition (\ref{def+C+labta}) of $C(\lambda)$. So, by (\ref{eq:generl-vector}), we define  assign
\[
 K_\nu
 \in V_{m-1,\wt(\lambda)}\cap\ker\mathcal D.
\]
So we take $v_{\lambda,i}:=K_\nu$. This completes the claim.
\end{proof}

Finally, we show that different $v_{\lambda,i}$ has the different leading monomials, which implies that they are different. Let
$\nu=T_{h_i(\lambda)}(\lambda)$ be an assigned split. Since
$\ell(\nu)=\ell(\lambda)+1$, the leading monomial of $K_\nu$
differs from that of the first-copy polynomial $K_\lambda$.
For a first-copy polynomial indexed by $\lambda'$ ($\ne\lambda$), we have
\[
 H(\nu)
 =\rad\bigl(\gcd(r,\lambda_1,\ldots,
                    \lambda_{\ell(\lambda)},h_i(\lambda))\bigr)
 =\gcd(H(\lambda),h_i(\lambda))=1,
\]
whereas $H(\lambda')>1$. The latter follows from
Lemma~\ref{lem:multiplicity}(i) when $\lambda'$ is nonempty, and from
$H(\varnothing)=\rad (r)>1$ otherwise. Thus a split polynomial has
a different leading monomial from every first-copy polynomial.

For a fixed $\lambda$ and $2\le i<j\le M(\lambda)$, the smallest
entries of $T_{h_i(\lambda)}(\lambda)$ and
$T_{h_j(\lambda)}(\lambda)$ are
$h_i(\lambda)$ and $h_j(\lambda)$, respectively. Since they are distinct,
the corresponding leading monomials are distinct. Finally, splits
arising from distinct source partitions cannot coincide: by
Lemma~\ref{lem:split}, their common image would determine a unique
source partition. Hence all assigned leading monomials are distinct.

\smallskip\noindent\textbf{Case 2.} $H=6$.

We first determine the number of additional copies. Since
$6\mid r$, $r\mid m$ and $m\ne6$, we have $m\ge12$. Formula
\eqref{eq:partition-multiplicity} becomes
\begin{eqnarray}\label{eq+M+lanbt+case}
M(\lambda)
 =\mathbf1_{2(\ell(\lambda)+1)\le m}
  +\mathbf1_{3(\ell(\lambda)+1)\le m}
  -\mathbf1_{6(\ell(\lambda)+1)\le m}.
\end{eqnarray}
It follows from $M(\lambda)\ge2$ and (\ref{eq+M+lanbt+case}) that
\begin{eqnarray}\label{ineq+M+lanbta}
 M(\lambda)=2,\quad
 \frac m6\le\ell(\lambda)\le\frac m3-1.
\end{eqnarray}

In what follows we divide the remaining proof into two cases:

\textbf{Case 2.1}: $\lambda_1\ge12$. So we can use a split. Take $h=5$ and let 
$\nu=T_5(\lambda)$. We have inequalities
\[
 0<h<6,\quad \lambda_1\ge12>2h,\quad 6\nmid 2h.
\]
Obviously, these inequalities satisfy the conditions of Lemma~\ref{lem:split}. So, by computing for $H=6$, we have every entry of $\nu$ lies between
$5$ and $n-m$, and
\[
 C(\nu)=2(\ell-1)+2\times 3=2\ell (\lambda)+4
       \le\frac{2m}{3}+2\le m-1,\quad
 \ell(\nu)=\ell(\lambda)+1
       \le\frac m3\le m-1.
\]
Therefore $\nu\in\mathcal P$, and we can define $K_\nu$ by (\ref{eq:generl-vector}). So, taking $v_{\lambda,2}=K_\nu$, we have
\[
 v_{\lambda,2}
 \in V_{m-1,\wt(\lambda)}\cap\ker\mathcal D.
\]
Here $H(\nu)=\gcd(6,5)=1$. As in Case~1,
by Lemma~\ref{lem:split}, we distinguish this leading partition from
all other assigned splits.

\textbf{Case 2.2}: $\lambda_1<12$. Since every entry is divisible
by $6$, we have $\lambda=(6^{\ell(\lambda)})$.
As $n-m\ge m\ge12$, we can define $f_6$ and $\Delta$ in (\ref{eq:even-vector}) and (\ref{eq:delta-vector}), respectively.
Set
\[
 v_{\lambda,2}
 :=x_0^{\,m/3-1}
   \Delta^{\,m/3-\ell(\lambda)}
   f_6^{\,2\ell(\lambda)-m/3}.
\]
In the following, we will prove $ v_{\lambda,2}\in \ker\mathcal D $.
The bounds in (\ref{ineq+M+lanbta}) on $\ell(\lambda)$ show that all exponents are
nonnegative integers. From Section~\ref{sec:kernel}, we have
$\mathcal D x_0=\mathcal D\Delta=\mathcal D f_6=0$;
the derivation rule therefore gives
$\mathcal Dv_{\lambda,2}=0$. Moreover,
\[
 \begin{aligned}
\deg( v_{\lambda,2})&= (m/3-1)+4(m/3-\ell(\lambda))
       +2(2\ell(\lambda)-m/3)=m-1,\\
wt( v_{\lambda,2})&= 6(m/3-\ell(\lambda))
       +6(2\ell(\lambda)-m/3)
       =6\ell(\lambda)=\wt(\lambda).
 \end{aligned}
\]
Consequently,
$v_{\lambda,2}\in V_{m-1,\wt(\lambda)}$.
Its leading term is
\[
 \operatorname{LT}(v_{\lambda,2})
 =2^{\,2\ell(\lambda)-m/3}
  x_0^{\,2m/3-1}
  x_3^{\,2m/3-2\ell(\lambda)}
  x_6^{\,2\ell(\lambda)-m/3}.
\]
Its leading partition is
\[
 \rho_\lambda
 :=(6^{\,2\ell(\lambda)-m/3},
    3^{\,2m/3-2\ell(\lambda)}).
\]
The multiplicity of $3$ is at least $2$. Moreover,
\[
 \ell(\rho_\lambda)=m/3\le m-1,\quad
 (\rho_\lambda)_1\le6\le n-m.
\]
Thus $\rho_\lambda\in\mathcal P$ and $H(\rho_\lambda)=3$.
Since $3(\ell(\rho_\lambda)+1)=m+3>m$,
By \eqref{eq:partition-multiplicity}, we have
$M(\rho_\lambda)=0$. This leading partition differs from those
of the first polynomials, which have positive multiplicity, and
from those of the split polynomials, which have $H=1$.
Different values of $\ell(\lambda)$ give different multiplicities
of $3$ in $\rho_\lambda$, so these exceptional leading
monomials are distinct as well.

This completes the proof. 
\end{proof}

\subsection{Positivity for \texorpdfstring{$m=6$}{m=6} from the KOH identity}
\label{sec:six}

For the nonnegativity of $R_k$, by Proposition~\ref{prop:domination}, it reduces to the remaining cases (i) $m=6$ and $r\in\{2,3\}$, (ii)
$m=6$ and $6\mid n$. In this subsection, we will complete the proofs of the remaining cases.

\begin{prop}\label{prop:domination+m6+2+3}
If $m=6$ and $r\in\{2,3\}$, then
$ C_k\ge R_k$ for $0\le k\le\delta/2.$
\end{prop}

\begin{proof}
Let $m=6$ and $r\in\{2,3\}$. The following proof is similar to that of $m\neq6$. For every nonempty
$\lambda\in\mathcal P$, we have
$H(\lambda)\in\{1,r\}$ and
$M(\lambda)\in\{0,1\}$; also $M(\varnothing)=1$.
If $M(\lambda)=1$, then by Lemma~\ref{lem:multiplicity} (i), we obtain
\[
 \begin{cases}
 \ell(\lambda)\le2,\quad
 C(\lambda)=2\ell(\lambda)\le4,&r=2,\\
 \ell(\lambda)\le1,\quad
 C(\lambda)\le3\ell(\lambda)\le3,&r=3.
 \end{cases}
\]
So $C(\lambda)\leq5=m-1$. Then we can define a $K_\lambda$ with
$v_{\varnothing,1}=x_0^5$ in terms of (\ref{eq:generl-vector}). Hence, take $v_{\lambda,1}:=K_\lambda$. These polynomials have distinct
leading monomials. So $C_k\ge R_k$ for $0\le k\le\delta/2.$
\end{proof}

For the case $m=6$ and $6\mid n$, we will use the KOH identity 
\cite{OHara1990,Zanello2015,Zeilberger1989}.

\begin{thm}[KOH identity]\label{thm:KOH}
Let $a\ge b\ge2$ be integers. For $\lambda\vdash b$, set
$ Y_j=\sum_{i=1}^j\lambda_i$ for $j\ge0.$
Then
\begin{equation}\label{eq:KOH}
 \qbinom{a+b}{b}
 =\sum_{\lambda\vdash b}
 q^{2\sum_{i\ge1}\binom{\lambda_i}{2}}
 \prod_{j=1}^{\ell(\lambda)}
 \qbinom{j(a+2)-Y_{j-1}-Y_{j+1}}{\lambda_j-\lambda_{j+1}}
\end{equation}
and every nonzero summand is symmetric and unimodal
about $ab/2$.
\end{thm}

\begin{prop}\label{prop:domination+m6+12}
Let$ n\geq 12$, $6\mid n$ and $r=\gcd(n,6)>1$.
Then $C_k\ge R_k$ for $0\le k\le5(n-6)/2$.
\end{prop}

\begin{proof}
We first compute $R(q)$, then obtain a lower bound for $C_k$ from
four summands of \eqref{eq:KOH}. We finish by comparing this bound
with $R_k$ in odd and even degrees.

The hypotheses imply $\gcd(n,m)=6$. In the divisor
formula (\ref{def+R}) for $R(q)$, the divisors greater than $1$ are $2,3,6$, with
M\"obius values $-1,-1,1$, respectively. Thus
\begin{equation}\label{eq:six-correction}
 R(q)=\qbinomd{n/2-1}{2}{2}+[n/3-1]_{q^3}-1.
\end{equation}
For the KOH identity, taking $a=n-6$ and $b=5$ yields
\[
 P(q)=\qbinom{n-1}{5}=\qbinom{(n-6)+5}{5}.
\]
We only need to use the four summands of \eqref{eq:KOH} indexed by
$(1^5)$, $(2,1^3)$, $(2,2,1)$, and $(3,1,1)$ to bound $C_k$.
They are
\[
\begin{array}{c|l}
 \lambda & \text{summand}\\ \hline
 (1^5)   &\qint{5n-29}\\
 (2,1^3) &q^2\qint{n-7}\qint{4n-25}\\
 (2,2,1) &q^4\qint{2n-15}\qint{3n-21}\\
 (3,1,1) &q^6\qbinom{n-8}{2}\qint{3n-21}.
\end{array}
\]
For example, if $\lambda=(3,1,1)$, then
$(Y_0,Y_1,Y_2,Y_3,Y_4)=(0,3,4,5,5)$ and
$(\lambda_1-\lambda_2,\lambda_3-\lambda_4)=(2,1)$,
which gives the last row. To compute their contributions to
$C_k$ for $0\le k\le\delta/2$, multiply each summand by $1-q$:
\[
\begin{aligned}
 (1-q)\qint{5n-29}
  &=1-q^{5n-29},\\
 (1-q)q^2\qint{n-7}\qint{4n-25}
  &=q^2\qint{n-7}-q^{4n-23}\qint{n-7},\\
 (1-q)q^4\qint{2n-15}\qint{3n-21}
  &=q^4\qint{2n-15}-q^{3n-17}\qint{2n-15},\\
 (1-q)q^6\qbinom{n-8}{2}\qint{3n-21}
  &=q^6\qbinom{n-8}{2}-q^{3n-15}\qbinom{n-8}{2}.
\end{aligned}
\]
For these four first differences, their negative terms begin above the range $0\le k\le5(n-6)/2$, since
$n\ge12$ and
\[
 \frac{5(n-6)}2
 <3n-17
 =\min\{5n-29,4n-23,3n-17,3n-15\}.
\]
So we next restrict these four first differences to consider their former terms.
Every remaining summand of \eqref{eq:KOH} has nonnegative first
differences in the same range by the symmetry and unimodality of Gaussian  polynomials. Hnece, for $0\le k\le5(n-6)/2$, we have
\begin{equation}\label{eq:six-bound}
 C_k\ge \mathbf1_{k=0}
  +\mathbf1_{2\le k\le n-6}
  +\mathbf1_{4\le k\le2n-12}
  +[q^{k-6}]\qbinom{n-8}{2}.
\end{equation}

We now compare this lower bound with the coefficients of $R(q)$.
By \eqref{eq:six-correction}, $\deg R(q)=2n-12$, so $R_k=0$ for
$k>2n-12$. As the right-hand side of \eqref{eq:six-bound} is
nonnegative, it remains to consider $0\le k\le2n-12$ in terms of the following two cases.

\textbf{Case 1}: For odd $k$, combining \eqref{eq:six-correction} and \eqref{eq:six-bound} gives
\[
 R_k=\mathbf1_{\{3\mid k,\ 1\le k\le n-6\}}
 \le\mathbf1_{2\le k\le n-6}\le C_k.
\]
For the first three even degrees, the same equations give
$c_0\ge1$, $c_2\ge1$, and $c_4\ge2$, while
$(r_0,r_2,r_4)=(1,1,2)$.

\textbf{Case 2}: For the remaining even degrees, write $k=2h$ with
$3\le h\le n-6$. We compute two Gaussian coefficients in
\eqref{eq:six-bound} and \eqref{eq:six-correction} by
\eqref{eq:coeff+Gauss}. In the two counts, the partitions have at most
two parts and are represented by $(2h-6-v,v)$ and $(h-v,v)$,
respectively, with zero parts omitted:
\begin{align*}
 [q^{2h-6}]\qbinom{n-8}{2}
 &=\bigl|\{v\in\Z:\max(0,2h-n+4)\le v\le h-3\}\bigr|\\
 &=\min(h-2,n-h-6)
   &&(3\le h\le n-7),\\
 [q^{2h}]\qbinomd{n/2-1}{2}{2}
 &=\bigl|\{v\in\Z:\max(0,h-n/2+3)
                     \le v\le\lfloor h/2\rfloor\}\bigr|\\
 &=\left\lfloor\frac{\min(h,n-h-6)}2\right\rfloor+1
   &&(0\le h\le n-6).
\end{align*}
The second count and \eqref{eq:six-correction} now give
\[
 R_{2h}=\left\lfloor\frac{\min(h,n-h-6)}2\right\rfloor+1
          +\mathbf1_{\{3\mid h,\ 2h\le n-6\}}
 \qquad(3\le h\le n-6).
\]
We compare this formula with \eqref{eq:six-bound} by the following three cases:

\textbf{Case 2.1:} If $3\le h\le n/2-3$, then 
$ C_{2h}\ge h \ge \lfloor h/2\rfloor+1+\mathbf1_{3\mid h}= R_{2h}.$

\textbf{Case 2.2:}  If $n/2-2\le h\le n-7$, then
\[
 C_{2h}\ge n-h-5
 \ge\left\lfloor\frac{n-h-6}{2}\right\rfloor+1
 =R_{2h}.
\]

\textbf{Case 2.3:} If $h=n-6$, then combining \eqref{eq:six-correction} and \eqref{eq:six-bound}
gives $C_{2n-12}\ge1=R_{2n-12}$.

This covers every integer $0\le k\le5(n-6)/2$. The proof is complete.
\end{proof}

\section{Positivity of coefficients of $W(q)$}
\label{sec:reciprocal-positive}

In this section, we will prove $W(q)\ge_q0$.

\begin{prop}\label{prop:W-positive}
Let $n,m\in\Z$ satisfy $3\le m\le n/2$ and $r=\gcd(n,m)>1$.
Then $W(q)\ge_q0$.
\end{prop}

\begin{proof}
Recall $\delta=(n-m)(m-1)$. Note by Lemma~\ref{lem:polynomiality} that $\deg W(q)=\delta-n-m$. If $\delta-n-m<0$, then the polynomial $W(q)$ must be zeros. If $\delta-n-m=0$, then $m=3$ and $n=9$, or $m=4$ and $n=8$. In this case, it is easy to obtain that $W(q)$ is a nonnegative constant. 
So in the following, we assume $\delta-n-m>0$. By Lemma~\ref{lem:polynomiality}, $W(q)=q^{\delta-n-m}W(q^{-1})$.
In consequence, in order to prove $W(q)\ge_q0$, it suffices to show that $[q^k]W(q)$ is nonnegative for $0\leq k\leq (\delta-n-m)/2$.
By (\ref{def+R}), we rewrite
$W(q)$ as
\begin{eqnarray}\label{W+rewrite}
 W(q)=\frac{R(q)}{1-q^m}
 +\sum_{\substack{d\mid r,d>1}}
   \mu(d)\frac{q^{L_d}R_d(q)}{1-q^m}.
\end{eqnarray}
By (\ref{Ld+posi}), we have $L_d\geq0$, where equality holds if and if $d=2$ and $m=4$ for every $d\mid r$ with $d>1$,
\begin{eqnarray}
L_d-\frac{(\delta-n-m)}{2}
 =\frac{m(d-2)(n-m)}{2d}\ge0
\end{eqnarray}
with equality if and only if $d=2$.
Since $R_d(0)=1$, the factor
$q^{L_d}R_d(q)/(1-q^m)$ has the lowest degree $L_d$.
Now we consider nonnegativity of $[q^k]W(q)$ by the following two cases:

\textbf{Case 1}: $0\le k<\frac{\delta-n-m}{2}$.
Combining (\ref{W+rewrite}) and Proposition~\ref{prop:R-positive} gives
\[
 [q^k]W(q)=\sum_{j\ge0}R_{k-jm}\ge0.
\]

\textbf{Case 2}: $k=\frac{\delta-n-m}{2}$.
If $d>2$, then $L_d-\frac{(\delta-n-m)}{2}>0$. Thus, in the second summation of (\ref{W+rewrite}), only the term with $d=2$ contributes in degree $L_2$.
Note $\mu(2)=-1$ and $R_2(0)=1$. It follows from (\ref{W+rewrite}) and Proposition~\ref{prop:R-positive} that
\begin{equation}\label{eq:midpoint-compensation}
 [q^{L_2}]W(q)=\begin{cases}
    \sum_{j\ge0}R_{L_2-jm}\geq0,&2\nmid r,\\[2pt]
    \displaystyle \sum_{j\ge0}R_{L_2-jm}-1\ge R_{L_2}-1,&2\mid r.
  \end{cases} 
\end{equation}
In the following we will prove $R_{L_2}\geq1$.

 We first consider the case $n-m>6$. Note for every $d\mid r$ with $d\ge3$ that
\[
 \begin{aligned}
 L_2-\deg R_d
 =m(n-m)\left(\frac12-\frac1d\right)-m\ge m\left(\frac{n-m}{6}-1\right)>0.
 \end{aligned}
\]
So the polynomial $R_d(q)$ has no the term $q^{L_2}$ for $d\geq3$.
It follows that $$R_{L_2}=[q^{L_2}]R(q)=[q^{L_2}]R_2(q)=[q^{L_2}]\left.\begin{bmatrix}n/2-1\\ m/2-1\end{bmatrix}\right|_{q\rightarrow q^2}=[q^{L_2/2}]\qbinom{n/2-1}{m/2-1}.$$
 By $d=2$, we have $m,n-m\in2\Z$, which implies $L_2\in 2\Z_{\ge0}.$ Moreover, we also have
$ 0\le\frac{L_2}{2} \le\deg \qbinom{n/2-1}{m/2-1}.$
In consequence, we obtain
\[
R_{L_2}=[q^{L_2/2}]\qbinom{n/2-1}{m/2-1}\ge1
\] since the coefficients of Gaussian polynomials are positive integers.

If $n-m\le6$, then $4\le m\le n-m\le6$ and both $m$ and $n-m$
are even. Thus $(n,m)\in\{(8,4),(10,4),(12,6)\}$, and direct
coefficient extraction gives
\[
 \begin{array}{c|c|c|c}
 (n,m)&R(q)&L_2&r_{L_2}\\ \hline
  (8,4)&[3]_{q^2}&0&1\\
  (10,4)&[4]_{q^2}&2&1\\
  (12,6)&\qbinomd{5}{2}{2}+[3]_{q^3}-1&6&3
 \end{array}
\]
In consequence, we have $R_{L_2}\ge1$ and
\eqref{eq:midpoint-compensation} is nonnegative. The proof is complete.
\end{proof}

\section{Proof of main results}
\label{sec:completion}
In this section, we will present the proofs of Theorem \ref{thm:strong} and Theorem \ref{thm:main}.
Before them, we first present a result for $G_{n,m}(q)$ for $m$ and $n$ being coprime (see Andrews \cite[Theorem~2]{Andrews} and
Haiman \cite[Section~2.5]{Haiman}). We give a self-contained proof as follows. 

\begin{lem}\label{lem:coprime}
Let $n$ and $m$ be integers with $0<m<n$ and $\gcd(n,m)=1$.
Then $G_{n,m}(q)\ge_q0$.
\end{lem}

\begin{proof}
For $m=1$, by (\ref{def+Fnm+Gnm}), we have
$G_{n,1}(q)=1/(1-q^{n+1})$, whose coefficients are obviously nonnegative.
Suppose $m>1$. By
$P(q)=\qbinom{n-1}{m-1}$, we define a rational function
\begin{eqnarray}
 B(q)=\frac{P(q)}{\qint m}.
\end{eqnarray}
It follows from (\ref{def+Fnm+Gnm}) that
\begin{eqnarray}
  G_{n,m}(q)=\frac{B(q)}{1-q^{n+1}}.
\end{eqnarray}

So, in order to prove that $G_{n,m}(q)\ge_q0$, it suffices to show $B(q)\ge_q0$. In what follows we will prove that $B(q)$ is a polynomial with only nonnegative coefficients.

If $\zeta^m=1$ and $\zeta\ne1$, then the condition $\gcd(n,m)=1$ implies that the order of $\zeta$ divides
$m$ but not $n$. By Lemma~\ref{lem:root-evaluation},
$P(\zeta)=0$. These roots are precisely simple of $\qint m$.
Thus we obtain $\qint m\mid P(q)$ and $B(q)\in\Z[q]$.

Easy computations give $B(0)=1$,
\begin{eqnarray}
 \deg B(q)=(m-1)(n-m-1)\ge0,
 \qquad q^{\deg B(q)}B(q^{-1})=B(q).
\end{eqnarray}
Thus, in order to prove that all coefficients of $B(q)$ are nonnegative, we only need to show that $[q^k]B(q)$ is nonnegative for $0\le k\le\deg B(q)/2$.
Since $\qint m=(1-q^m)/(1-q)$, we have
\begin{eqnarray}\label{coeff+B}
[q^k]B(q)=[q^k]\frac{(1-q)P(q)}{1-q^m}
 =\sum_{j\ge0}[q^{k-jm}](1-q)P(q)
\end{eqnarray}
for $k\ge0$.
For $0\le k\le\deg B(q)/2$, each nonnegative index in this sum satisfies
\begin{eqnarray}
 0\le k-jm\le k\le\frac{\deg B(q)}{2}
 \le\frac{(m-1)(n-m)}{2}=\frac{\deg P(q)}{2}.
\end{eqnarray}

Hence, applying Lemma~\ref{lem:kernel} with $a=n-m$ and $b=m-1$ to (\ref{coeff+B}) yields nonnegativity of each term. So $[q^k]B(q)$ is nonnegative for $0\le k\le\deg B(q)/2$. We complete the proof.
\end{proof}

\begin{proof}[\textbf{Proof of Theorem ~\ref{thm:strong}}]
From symmetry of Gaussian polynomials, we have $G_{n,m}(q)=G_{n,n-m}(q)$. So it suffices to
consider $3\le m\le n/2$.

If $\gcd(n,m)=1$, then
$G_{n,m}(q)\ge_q0$ by Lemma~\ref{lem:coprime}.
So let $r=\gcd(n,m)>1$. By Proposition~\ref{prop:reduction}, in order to prove $G_{n,m}(q)\ge_q0$, it suffices to show
\[
  R(q)\ge_q0,\quad W(q)\ge_q0\quad and \quad C_k-R_k\ge0\quad(0\le k\le\delta/2).
\]  The first two nonnegativity $R(q)\ge_q0$ and $W(q)\ge_q0$ follows from Propositions~\ref{prop:R-positive}
and~\ref{prop:W-positive}, respectively. For the third nonnegativity, it follows from Proposition~\ref{prop:domination} Propostion \ref{prop:domination+m6+2+3} and Proposition~\ref{prop:domination+m6+12}. In consequence, we obtain $G_{n,m}(q)\ge_q0$.
\end{proof}

\begin{proof}[\textbf{Proof of Theorem~\ref{thm:main}}]
For $m\ge2$, by \eqref{eq:f-to-g}, we have
\[
 F_{n,m}(q)=G_{n,m}(q)
 \prod_{j=n+2}^{n+m-1}\frac1{1-q^j},
 \qquad
 \frac1{1-q^j}=\sum_{i\ge0}q^{ij}\ge0\qquad(j\ge1).
\]
Hence, if $3\le m\le n-3$, we have $ F_{n,m}(q)\ge_q0$ by Theorem~\ref{thm:strong}. The following is the boundary cases:

(i) If $m=n-1\ge2$, then $\gcd(n,m)=1$, so
$F_{n,m}(q)\ge_q0$ by Lemma~\ref{lem:coprime}. 

(ii) If $m=1$, then obviously $F_{n,1}(q)=1$.

(iii) If $m=2$ and $n$ is odd, then direct cancellation yields
\begin{equation}\label{eq:m2}
 F_{n,2}(q)=G_{n,2}(q)
 =\frac{1-q^{n-1}}{(1-q^2)(1-q^{n+1})}
\end{equation}
whose coefficients are nonnegative by Lemma~\ref{lem:coprime}.

(iv) If $m=n-2$ and $n$ is odd, then $ F_{n,n-2}(q)=G_{n,n-2}(q)=G_{n,2}(q)$. Its coefficients are nonnegative by (iii).
This completes the proof of Theorem \ref{thm:main}.
\end{proof}

\begin{rem}
The exclusions for even \(n\) are necessary. For even \(n\geq4\), by \eqref{eq:m2}, we have $$[q^{n-1}]F_{n,2}(q)=[q^{n-1}]F_{n,n-2}(q)=[q^{n-1}]=G_{n,2}(q).$$
\end{rem}



\begin{thebibliography}{99}

\bibitem{And93}
G. E. Andrews, On a conjecture of Peter Borwein, J. Symbolic Comput., 20 (1995), 487--501.
Symbolic computation in combinatorics (Ithaca, NY, 1993).


\bibitem{Andrews}
G.~E.~Andrews,
\emph{The Friedman--Joichi--Stanton monotonicity conjecture at primes},
in \emph{Unusual Applications of Number Theory} (M.~B.~Nathanson, ed.),
DIMACS Ser. Discrete Math. Theoret. Comput. Sci., vol.~64,
Amer. Math. Soc., Providence, RI, 2004, pp.~9--15.


\bibitem{ChenJia2022}
W. Y. C. Chen, I. D. D. Jia,
Semi-invariants of binary forms and Sylvester's theorem,
\emph{Ramanujan J.} \textbf{59} (2022), no.~1, 297--311.

\bibitem{FJS1994}
J. Friedman, J. T. Joichi, D. Stanton,
More monotonicity theorems for partitions,
\emph{Exp. Math.} \textbf{3} (1994), no.~1, 31--37.

\bibitem{Haiman}
M.~D.~Haiman,
\emph{Conjectures on the quotient ring by diagonal invariants},
J.~Algebraic Combin. \textbf{3} (1994), no.~1, 17--76.


\bibitem{Han}
G. Han, personal communication, 2002.


\bibitem{OHara1990}
K. M. O'Hara,
Unimodality of Gaussian coefficients: a constructive proof,
\emph{J. Combin. Theory Ser. A} \textbf{53} (1990), no.~1, 29--52.



\bibitem{PakPanova2014}
I. Pak, G. Panova,
Unimodality via Kronecker products,
\emph{J. Algebraic Combin.} \textbf{40} (2014), no.~4, 1103--1120.


\bibitem{PakPanova2017}
I. Pak, G. Panova,
Bounds on certain classes of Kronecker and $q$-binomial coefficients,
\emph{J. Combin. Theory Ser. A} \textbf{147} (2017), 1--17.


\bibitem{PakPanovaVallejo2016}
I. Pak, G. Panova, E. Vallejo,
Kronecker products, characters, partitions, and the tensor square conjectures,
\emph{Adv. Math.} \textbf{288} (2016), 702--731.

\bibitem{PS2003}
T. Prellberg, D. Stanton,
Proof of a monotonicity conjecture,
\emph{J. Combin. Theory Ser. A} \textbf{103} (2003), no.~2, 377--381.


\bibitem{Stanley1980}
R. P. Stanley,
Weyl groups, the hard Lefschetz theorem, and the Sperner property,
\emph{SIAM J. Algebraic Discrete Methods} \textbf{1} (1980), no.~2, 168--184.

\bibitem{StanleyEC1}
R. P. Stanley,
\emph{Enumerative Combinatorics}, Vol.~1, second ed.,
Cambridge University Press, Cambridge, 2012.

\bibitem{Wang2022}
C. Wang,
An analytic proof of the Borwein Conjecture,
\emph{Adv. Math.} \textbf{394} (2022), Paper No.~108028.
\bibitem{Zanello2015}
F. Zanello,
Zeilberger's KOH theorem and the strict unimodality of $q$-binomial coefficients,
\emph{Proc. Amer. Math. Soc.} \textbf{143} (2015), no.~7, 2795--2799.

\bibitem{Zeilberger1989}
D. Zeilberger,
Kathy O'Hara's constructive proof of the unimodality of the Gaussian polynomials,
\emph{Amer. Math. Monthly} \textbf{96} (1989), no.~7, 590--602.




\end{thebibliography}
\end{document}